\documentclass{amsart}
\usepackage{amsmath}
  \usepackage{paralist}
  \usepackage{graphics} %% add this and next lines if pictures should be in esp format
  \usepackage{epsfig} %For pictures: screened artwork should be set up with an 85 or 100 line screen
\usepackage{graphicx}  \usepackage{epstopdf}%This is to transfer .eps figure to .pdf figure; please compile your paper using PDFLeTex or PDFTeXify.
\usepackage{multirow}
\usepackage{booktabs}
\usepackage[pagewise]{lineno} %\linenumbers
 \usepackage[colorlinks=true]{hyperref}
\hypersetup{urlcolor=blue, citecolor=red}

\newtheorem{theorem}{Theorem}[section]
\newtheorem{corollary}{Corollary}

\newtheorem{lemma}[theorem]{Lemma}

\theoremstyle{definition}
\newtheorem{definition}[theorem]{Definition}

\title[The Spectrum of DDEs with Hierarchical Large Delays] %Use the shortened version of the full title
      {The Spectrum of Delay Differential Equations with Multiple Hierarchical Large Delays}
      
\author{Stefan Ruschel}
\address{Department of Mathematics, University of Auckland, Auckland 1142, New Zealand}
\email{stefan.ruschel@auckland.ac.nz}

\author{Serhiy Yanchuk}
\address{Institut f\"{u}r Mathematik, Technische Universit\"{a}t Berlin, Strasse des 17. Juni 136, 10623 Berlin, Germany}
\email{yanchuk@math.tu-berlin.de}


\thanks{The authors acknowledge support by the Deutsche Forschungsgemeinschaft (DFG, German\\
Research Foundation) - Project 411803875 and SFB 910. The research was conducted while SR was doctoral student at Technische Universit\"{a}t Berlin.}

\begin{document}
\maketitle

%The abstract of your paper
\begin{abstract}
We prove that the spectrum of the linear delay differential equation
$x'(t)=A_{0}x(t)+A_{1}x(t-\tau_{1})+\ldots+A_{n}x(t-\tau_{n})$ with
multiple hierarchical large delays $1\ll\tau_{1}\ll\tau_{2}\ll\ldots\ll\tau_{n}$
splits into two distinct parts: the strong spectrum and the pseudo-continuous
spectrum. As the delays tend to infinity, the strong spectrum converges
to specific eigenvalues of $A_{0}$, the so-called asymptotic strong
spectrum. Eigenvalues in the pseudo-continuous spectrum however, converge
to the imaginary axis. We show that after rescaling, the pseudo-continuous
spectrum exhibits a hierarchical structure corresponding to the time-scales
$\tau_{1},\tau_{2},\ldots,\tau_{n}.$ Each level of this hierarchy
is approximated by spectral manifolds that can be easily computed.
The set of spectral manifolds comprises the so-called asymptotic continuous
spectrum. It is shown that the position of the asymptotic strong spectrum
and asymptotic continuous spectrum with respect to the imaginary axis
completely determines stability. In particular, a generic destabilization
is mediated by the crossing of an $n$-dimensional spectral manifold
corresponding to the timescale $\tau_{n}$. 
\end{abstract}

\section{Introduction }

Delay Differential Equations (DDE) are highly relevant in various
fields of applications including secure communication \cite{Argyris2005},
information processing \cite{Appeltant2011}, and many others \cite{Erneux2009,Atay2010,Hartung2006}.
When studying these - generally nonlinear - equations close to equilibrium,
one is first concerned with the spectral properties of a corresponding
linearized system of the form \cite{Bellman1963,Hale1993,Diekmann1995,Sieber2017}
\begin{equation}
x'(t)=A_{0}x(t)+A_{1}x(t-\tau_{1})+\cdots+A_{n}x(t-\tau_{n}).\label{eq:DDE}
\end{equation}
A complete description of the spectrum of (\ref{eq:DDE}) can be formidable
task even for a single delay, and is generally unfeasible for two
or more. Specific cases therefore have been studied in much detail,
see \cite{Hayes1950,Cooke1986,Belair1994a,Lucht1998,Ruan2003} and
references therein. It is convenient however, if the involved time
delays are large. In this paper, we provide a detailed description
of the spectrum of (\ref{eq:DDE}) with finitely many large hierarchical
delays 
\begin{equation}
1\ll\tau_{1}\ll\tau_{2}\ll\dots\ll\tau_{n},\label{eq:hierarchical}
\end{equation}
which bear some analogy to spatially extended systems \cite{Yanchuk2015a,YanchukGiacomelli2017}.
Figure~\ref{fig:intro-example} provides an example of such a spectrum.
One can observe a complicated structure and that there is a large
number of eigenvalues that are very close to the imaginary axis, i.e.
they play important role for determining stability. This manuscript
provides not only an analytical description of such spectra, but also
explicit analytic expressions for their approximations.
\begin{figure}
\begin{centering}
\includegraphics[width=1.\linewidth]{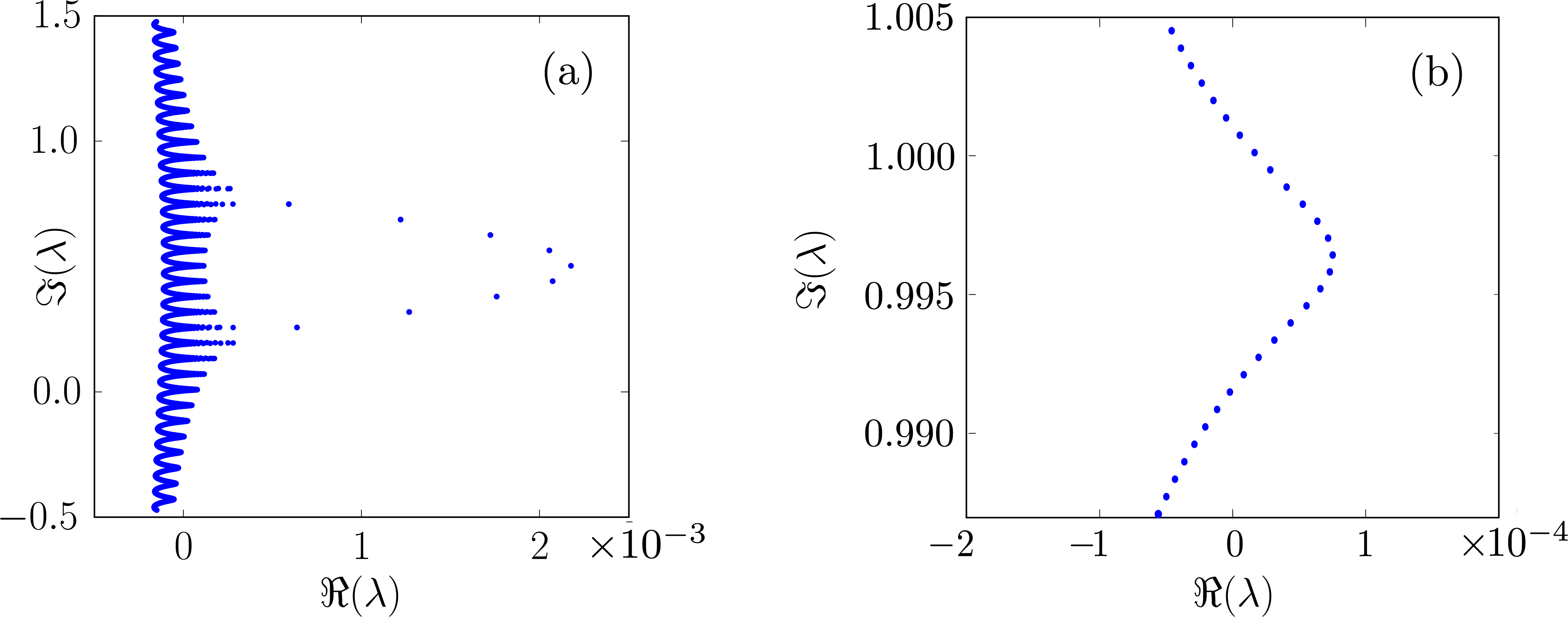}
\par\end{centering}
\caption{\label{fig:intro-example} Example of the numerically computed spectrum
of eigenvalues for system (\ref{eq:DDE}) with $n=2$, $A_{0}=-0.4+0.5i$,
$A_{1}=0.5$, $\tau_{1}=100$, and $\tau_{2}=10000$. Panel (a): blue dots are
numerically computed eigenvalues. Panel (b): zoom into panel (a).}
\end{figure}
Examples for DDEs with multiple large hierarchical delays can be drawn
from non-linear optics, where the finite-time communication delays
are typically much larger than the device's internal timescales \cite{Soriano2013}.
Specific examples of (\ref{eq:DDE}) for two hierarchical large delays
of different size include semiconductor lasers with two optical feedback
loops of different length \cite{Yanchuk2014a,Yanchuk2015,Marconi2015,Giacomelli2019},
and ring-cavity lasers with optical feedback \cite{Franz2008,Otto2012,Jaurigue2017}.
Additional examples can be found in applications to biological systems,
when a corresponding separation of time-scales is justified \cite{Shayer2000,Saha2017,Cooke1996,Ruschel2018}. 

This work extends the results of \cite{Lichtner2011} to multiple
large hierarchical delays, under more general non-genericity conditions.
We show that the spectrum splits into two distinct parts with different
scaling behavior: the strong spectrum and the pseudo-continuous spectrum.
As the delays tend to infinity, the strong~spectrum converges to
specific eigenvalues of $A_{0}$, the so-called asymptotic~strong~spectrum.
Eigenvalues in the pseudo-continuous spectrum converge to the imaginary
axis as the delays increase. We show that after rescaling the pseudo-continuous
spectrum exhibits a hierarchical structure corresponding to the time-scales
$\tau_{1},\tau_{2},\ldots,\tau_{n}.$ In particular, we show that
this set of eigenvalues can be represented as a union of subsets corresponding
to different timescales $\tau_{k}$. Generically, for $1\leq k\leq n$,
each of these sets can be associated with a $k$-dimensional spectral
manifold in the positive half plane that extends to the negative half-plane
under certain degeneracy conditions related to the rank of the matrices
$A_{k+1},\ldots,A_{n}$ or if $k=n$. These manifolds can be computed explicitly
and the corresponding eigenvalues $\lambda\in\mathbb{C}$ can be found
by projecting the manifolds to the complex plane. Moreover, the asymptotic
spectra are exact at the imaginary axis, and therefore, the stability
boundaries are completely determined by the position of the asymptotic
strong spectrum and asymptotic continuous spectrum with respect to
the imaginary axis. It is shown that a generic destabilization of
such a system takes place by the crossing of an $n$-dimensional spectral
manifold corresponding to the timescale $\tau_{n}$. Section~\ref{sec:overview}
contains an overview of our rigorous results, along with an introduction
to the needed basic concepts. The corresponding proofs contained in
Sec.~\ref{sec:proofs} are largely influenced by the proofs in \cite{Lichtner2011}.
Similar, to the single large delay case \cite{Heiligenthal2011},
our results in part can be extended to linear DDEs with time varying
coefficients, see Ref. \cite{DHuys2013} for more details.

Aiming at a rigorous description, the presentation in Sec.~\ref{sec:overview}
sometimes appears technical. We included Table \ref{tab:notation} for quick referencing of frequently used notation throughout the article. To illustrate our results and to foster understanding of the main ideas, we present an example of analytically
and numerically computed spectra for the scalar case with two large
hierarchical delays in Sec.~\ref{sec:example}.

\begin{table}%[p]
	%\begin{centering}	    
	\begin{tabular*}{\textwidth}{l@{\extracolsep{\fill}}lll}
		Symbol 	 &	Description  & Reference	\\
		\midrule
				
		$\Sigma^\varepsilon$ 		             
		& Spectrum 
		& Eq.~(\ref{eq:spec})\\
		
        $\Sigma_s^\varepsilon$ 		             
		& Strong spectrum 		                    
		& Def.~\ref{def:strong},~Eq.~(\ref{eq:strong-spec})\\
		
        $\Sigma_c^\varepsilon$ 		             
		& Pseudo-continuous spectrum 		                    
		& Def.~\ref{def:strong},~Eq.~(\ref{eq:pseudocont-spectrum})\\
		
		$\tilde{\Sigma}_{k}^{\varepsilon}$ 		             
		& Truncated stable $\tau_k$-spectrum						                    
		& Def.~\ref{def:deg-spec},~Eq.~(\ref{eq:truncated-stable-spectrum})\\
		
		\par
		\midrule
		\par
		
        $\mathcal{A}_0$
		& Asymptotic strong spectrum								                    
		& Def.~\ref{def:strong},~Eq.~(\ref{eq:asympt-strong-spec-1})\\
		
		$S_{0}^{+}$
		& Asymptotic strong unstable spectrum								                    
		& Def.~\ref{def:strong},~Eq.~(\ref{eq:asympt-strong-spec})\\
		
		$\tilde{S}_{0}^{-}$	             
		& Asymptotic strong stable spectrum							                    
		& Def.~\ref{def:deg-spec},~Eq.~(\ref{eq:asymptotic-strong-stable-spectrum})\\
		
		$\mathcal{A}_k$ 
		& Asymptotic continuous $\tau_k$-spectrum								                    
		& Def.~\ref{def:trunc-spec},~Eq.~(\ref{asymptotic-continuous-spectrum})\\
		
		$S_{k}^{+}$	
		& Asymptotic continuous stable $\tau_k$-spectrum									                    
		& Def.~\ref{def:trunc-spec},~Eq.~(\ref{eq:asymptotic-continuous-unstable-spec})\\
		
		$\tilde{S}_{k}^{-}$
		& Asymptotic continuous unstable $\tau_k$-spectrum						                    
		& Def.~\ref{def:trunc-spec},~Eq.~(\ref{asymptotic-continuous-stable-spec})\\
		
		\par
		\midrule
		\par
		
		$A_k$ 		             
		& Coefficient matrix corresponding to delay $\tau_k$							                    
		& Eq.~(\ref{eq:DDE})\\
		
		$A_{j,1}^{(k)}$ 		             
		& Projection of coefficient matrix $A_j$ to the\\ 
		& cokernels of matrices $A_l$, $l=k,k+1,\ldots,n$							                    
		& Eq.~(\ref{eq:coeff-matrix-projected})\\
		
		$\chi^\varepsilon(\lambda)$ 		
		& Characteristic function						    
		& Eq.~(\ref{eq:cheq})\\
		
		$\tilde\chi^\varepsilon_k(\lambda)$ 		
		& Projected characteristic equation, $0\leq k< n$ 
		& Def.~\ref{def:deg-spec},~Eq.~(\ref{eq:proj-ce})\\
		
		$\chi_k,\tilde{\chi}_k$	  
		& Truncated characteristic equation, $0\leq k< n$
		& Def.~\ref{def:trunc-spec},
		Eqs.~(\ref{eq:ce-t1})--(\ref{eq:ce-tk})\\
		\midrule						
	\end{tabular*}\linebreak{}
	%\par\end{centering}
	\caption{\label{tab:notation}Frequent notations.}
\end{table}

\section{Basic concepts and Overview of Results\label{sec:overview}}

We consider the special case of hierarchical time delays
$\tau_k = \sigma_k\varepsilon^{-k}$, where $\sigma_k>0$, $1\le k\le n $, and $\varepsilon>0$ is a small parameter. Hence, we consider the linear Delay Differential
Equation (DDE) 
\begin{equation}
x'(t)=A_{0}x(t)+\sum_{k=1}^{n}A_{k}x(t-\sigma_{k}\varepsilon^{-k})\label{eq:DDE-1}
\end{equation}
with $n\geq2$ hierarchical large delays and study the asymptotic
behavior of its solutions as $\varepsilon\to0$. Throughout, we assume
that $x(t)\in\mathbb{C}^{d}$ is a complex-valued, Euclidean vector
of size $d$ and $A_{k}\in\mathbb{C}^{d\times d},\,A_{k}\neq0\,,\,0\leq k\leq n$
are given matrices independent of time and $\varepsilon$. Existence
and uniqueness of solutions to (\ref{eq:DDE-1}), as well as the specific
notions of solution and state space will not be covered here, but
can be found in classic text books on Delay and Functional Differential
Equations \cite{Hale1993,Diekmann1995,Bellman1963}. 

Equation (\ref{eq:DDE-1}) can be thought of as similar in spirit
to an Ordinary Differential Equation (ODE) except that it may exhibit
so-called small solutions; those are solutions that ``collide''
with the trivial solution $x\equiv0$ in finite time, say $t_{1}$,
and equal zero for all $t\geq t_{1}$. Apart from this peculiarity,
that is up to small solutions, any solution of (\ref{eq:DDE-1}) can
be written as a superposition of exponential functions as in the case
of ODEs \cite{Hale1993}. In particular, the long term behavior of
the solution as $t\to\infty$ is governed by the characteristic exponents. 

In this sense, solving (\ref{eq:DDE-1}) is equivalent to finding
nontrivial solutions to the matrix-valued quasi-polynomial equation
$\Delta^{\varepsilon}(\lambda)v=0,$ where $\Delta^{\varepsilon}:\mathbb{C}\to\mathbb{C}^{d\times d},$
\begin{equation}
\Delta^{\varepsilon}(\lambda):=-\lambda I+A_{0}+\sum_{k=1}^{n}A_{k}\exp\left(-\lambda\sigma_{k}\varepsilon^{-k}\right)\label{eq:chma-1}
\end{equation}
is the characteristic matrix. A nontrivial solution $v$ exists, if
and only if there is $\lambda\in\mathbb{C}$ such that $\ker\Delta^{\varepsilon}(\lambda)\neq\emptyset$,
or equivalently $\det\Delta^{\varepsilon}(\lambda)=0$. For simplicity,
let us assume $\lambda$ is a simple root of $\det\Delta^{\varepsilon}(\lambda)$.
Together with a corresponding $0\neq v\in\ker\Delta^{\varepsilon}(\lambda)\subseteq\mathbb{C}^{d}$,
it gives rise to a solution $t\mapsto v\exp(\lambda t)$ of Eq.~(\ref{eq:DDE-1}).
See Ref.~\cite{Hale1993} for further details. The pair $(\lambda,v)\in\mathbb{C}\times\mathbb{C}^{d}$
is called an eigenvalue-eigenvector pair and the entirety of eigenvalues
$\lambda$ is called the spectrum 
\begin{equation}
\Sigma^{\varepsilon}:=\left\{ \lambda\in\mathbb{C}\,|\,\det\Delta^{\varepsilon}(\lambda)=0\right\} \label{eq:spec}
\end{equation}
of (\ref{eq:DDE-1}).

Hence, the problem consists of describing the asymptotic location
of complex-valued solutions to the characteristic equation
\begin{equation}
\chi^{\varepsilon}(\lambda):=\det\Delta^{\varepsilon}(\lambda)=0\label{eq:cheq}
\end{equation}
as $\varepsilon\to0$. For each fixed $\varepsilon>0$, much is known
about the solutions of (\ref{eq:cheq}). Firstly, there are countably
many solutions that continuously depend on parameters. Secondly, the real parts of solutions
accumulate at $-\infty$. Within each vertical stripe $[\alpha,\beta]\times i\mathbb{R}\subseteq\mathbb{C}$
there are only finitely many solutions \cite{Bellman1963,Hale1993}.
In particular, $\beta$ can be chosen $+\infty$ \cite{Hale1993}. The
following Secs.~\ref{subsec:Low-Rank-perturbations}--\ref{subsec:Spectral-manifolds}
present our main results. At first, it is convenient to discuss the non-generic case when some of the matrices $A_k$ do not have full rank, starting from highest order $A_n$. In this case, one can immediately identify spectral subsets of truncated characteristic equations that approximate certain subsets of $\Sigma^{\varepsilon}\cap\{\lambda\in\mathbb{C}|\Re(\lambda)<0\}$ for sufficiently small $\varepsilon$.

\subsection{Degeneracy spectrum\label{subsec:Low-Rank-perturbations}}

From the point of view of applications, we certainly cannot expect
the matrices $A_{k},\,0\leq k\leq n$ to be invertible. In this section,
we introduce the necessary conditions for our main Theorem \ref{thm:spec-approx} to hold.
To set the stage, consider the case when $A_{n}$ is not invertible.
Then, for sufficiently small $\varepsilon,$ we may think of Eq.~(\ref{eq:cheq}),
as a low rank perturbation of a certain truncated characteristic equation,
see Theorem~\ref{thm:degeneracy-spec}. Let us explain. If $d_{n}:=\text{rank}A_{n}<d$,
there exist unitary matrices $U_{n},V_{n}$ such that 
\begin{equation}
A_{n}=U_{n}\left(\begin{array}{cc}
0 & 0\\
0 & A_{n,4}^{(n)}
\end{array}\right)V_{n}^{\ast},\label{eq:singular-value-decomp}
\end{equation}
where $A_{n,4}^{(n)}\in\mathbb{C}^{d_{n}\times d_{n}}$ is a diagonal
matrix of full rank, and $V_{n}^{\ast}$ is the conjugate transpose
of $V_{n}$. Equation~(\ref{eq:singular-value-decomp}) is the singular
value decomposition of $A_{n}$ and the columns of $U_{n}$ and $V_{n}$
are the left and right singular vectors of $A_{n}$, respectively.

The columns of the matrices $U_{n}=[U_{n,1},U_{n,2}]$ and $V_{n}=[V_{n,1},V_{n,2}]$
are the left and right singular
vectors corresponding to the cokernel ($U_{n,1}$ and $V_{n,1}$)
and image ($U_{n,2}$ and $V_{n,2}$) of $A_{n}$. In particular, 
 $U_{n,1}^{\ast}A_{n}V_{n,1}=0$ and $U_{n,2}^{\ast}A_{n}V_{n,2}=A_{n,4}^{(n)}$, 
correspond to the projection onto the cokernel and image of $A_{n}$, respectively. This projection allows to define the following spectral sets. 
\begin{definition}[non-generic spectral subsets]
\label{def:deg-spec}Let $d_{n}:=\text{rank}A_{n}<d$.
\begin{enumerate}
\item[(i)] Define $U_{n,1},V_{n,1}\in\mathbb{C}^{d\times(d-d_{n})}$ as the
matrices containing the left and right singular vectors of $A_{n}$
corresponding to the singular value zero. Denote
\[
J_{1}^{(n)}:=U_{n,1}^{\ast}V_{n,1},
\]
\[
A_{j,1}^{(n)}:=U_{n,1}^{\ast}A_{j}V_{n,1},\quad j=0,\dots,n-1,
\]
and the corresponding \textit{projected characteristic equation}
\begin{equation}\label{eq:proj-ce}
\tilde{\chi}_{n-1}^{\varepsilon}(\lambda):=\det\left(-\lambda J_{1}^{(n)}+A_{0,1}^{(n)}+\sum_{k=1}^{n-1}A_{k,1}^{(n)}\exp\left(-\lambda\sigma_{k}\varepsilon^{-k}\right)\right).
\end{equation}
The set 
\[
\tilde{\Sigma}_{n-1}^{\varepsilon}:=\left\{ \lambda\in\mathbb{C}\,|\,\tilde{\chi}_{n-1}^{\varepsilon}(\lambda)=0,\,\Re(\lambda)<0\right\} 
\]
is called the \textit{truncated stable $\tau_{n-1}$-spectrum}.
\item[(ii)] If $A_{n-1,1}^{(n)}$ is again not invertible, this procedure is applied
iteratively. Recursively for all $1\leq k\leq n-1$ (starting from $n-1$), if
$\det A_{k,1}^{(k+1)}=0$, define $\tilde{U}_{k,1},\tilde{V}_{k,1}$
(notice the tilde notation) containing left and right singular vectors
of $A_{k,1}^{(k+1)}$ corresponding to the singular value zero. Denote
\[
J_{1}^{(k)}:=\tilde{U}_{k,1}^{\ast}J_{1}^{(k+1)}\tilde{V}_{k,1},
\]
\begin{equation}\label{eq:coeff-matrix-projected}
A_{j,1}^{(k)}:=\tilde{U}_{k,1}^{\ast}A_{j,1}^{(k+1)}\tilde{V}_{k,1},\quad j=0,\dots,k-1
\end{equation}
 and the corresponding truncated characteristic equation
\[
\tilde{\chi}_{k-1}^{\varepsilon}(\lambda):=\det\left(-\lambda J_{1}^{(k)}+A_{0,1}^{(k)}+\sum_{j=1}^{k-1}A_{j,1}^{(k)}\exp\left(-\lambda\sigma_{j}\varepsilon^{-j}\right)\right),
\]
for $1\leq k < n-1$, and
\[
\tilde{\chi}_{0}(\lambda):=\det\left(-\lambda J_{1}^{(1)}+A_{0,1}^{(1)}\right).
\]
\item[(iii)] Define $\underline{k},$ $1\leq\underline{k}\leq n-1$ as the smallest
index such $\det A_{k,1}^{(k+1)}=0$ for all $\underline{k}\leq k\leq n-1$
and $\det A_{n}=0$.
\item[(iv)] For $1\leq k\leq n-1,$ define set 
\begin{equation}\label{eq:truncated-stable-spectrum}
\tilde{\Sigma}_{k}^{\varepsilon}:=\left\{ \lambda\in\mathbb{C}\,|\,\tilde{\chi}_{k}^{\varepsilon}(\lambda)=0,\,\Re(\lambda)<0\right\} 
\end{equation}
for $k\geq\underline{k}$, and $\tilde{\Sigma}_{k}^{\varepsilon}=\emptyset$
otherwise. The set $\tilde{\Sigma}_{k}^{\varepsilon}$ is called the
\textit{truncated stable $\tau_{k}$-spectrum}. If $\underline{k}=1,$ set
\begin{equation}\label{eq:asymptotic-strong-stable-spectrum}
\tilde{S}_{0}^{-}:=\left\{ \lambda\in\mathbb{C}\,|\,\tilde{\chi}_{0}(\lambda)=0,\,\Re(\lambda)<0\right\} ,
\end{equation}
and $\tilde{S}_{0}^{-}:=\emptyset$ otherwise. $\tilde{S}_{0}^{-}$
is called \textit{asymptotic strong stable spectrum}.
\item[(v)] If $\underline{k}=1$ and $\det J_{1}^{(1)}=0$, define the matrices
$\mathcal{U}_1,\mathcal{V}_1$ containing left and right singular vectors
of $J_{1}^{(1)}$ corresponding to the singular value zero. 
\end{enumerate}
\end{definition}

These sets correspond to spectral directions along which Eq.~(\ref{eq:DDE})
acts as a DDE with fewer delays or even an ODE. Before stating our result, we have to guarantee that
Eq.~(\ref{eq:DDE}) is indeed a DDE and cannot be transformed into
a system of ODEs through variable transformations, one has to demand
the following non-degeneracy condition.

\medskip{}
{\noindent}\textbf{Condition (ND).} If $\det$$A_{n}=0$, $\underline{k}=1$
and $\det J_{1}^{(1)}=0,$ then $\det\left(\mathcal{U}_1^{\ast}A_{0,1}^{(1)}\mathcal{V}_1\right)\neq0$.

\medskip{}
\hspace{-11pt}This is a rather abstract condition. In order to build
some intuition, consider the following example. Let $d=2$, $n=1,$ and the matrices $A_{0}$
and $A_{1}$ are given by 
\[
A_{0}=\left(\begin{array}{cc}
a_{1} & a_{2}\\
a_{3} & a_{4}
\end{array}\right),\quad A_{1}=\left(\begin{array}{cc}
0 & 1\\
0 & 0
\end{array}\right).
\]
Clearly, $\text{rank}A_{1}=1<2$  and one readily
computes 
\[
U_{1,1}=\left(\begin{array}{c}
0\\
1
\end{array}\right),\quad V_{1,1}=\left(\begin{array}{c}
1\\
0
\end{array}\right),
\]
as well as
\[
J_{1}^{(1)}=U_{1,1}^{\ast}V_{1,1}=0,
\quad A_{0,1}^{(1)}=U_{1,1}^{\ast}A_0V_{1,1}=a_3. 
\]
We may set
$\mathcal{U}_1=1,$ $\mathcal{V}_1=1$. Condition (ND) then reads $$\mathcal{U}_1^{\ast}A_{0,1}^{(1)}\mathcal{V}_1 = a_{3}\neq0.$$ If $a_{3}=0$, the system is degenerate;
it corresponds to an ODE. Straightforward computation shows
that the (general) characteristic equation 
$$ 0=(a_{1}-\lambda)(a_{4}-\lambda)-a_{3}(a_{2}+e^{-\lambda\sigma_{1}/\varepsilon}) $$
does not depend on $e^{-\lambda\sigma_{1}/\varepsilon}$ in this case,
and the spectrum consists of $\{a_{1},a_{4}\}$ for all $\varepsilon$. 
On the basis of Def.~\ref{def:deg-spec},  the following Theorem \ref{thm:degeneracy-spec} provides a hierarchical approximation of $$\Sigma^{\varepsilon}\cap\{\lambda\in\mathbb{C}|\Re(\lambda)<0\}$$ by spectral subsets of truncated characteristic equations $\tilde\chi_k^\varepsilon$, when some of the matrices $A_k$ do not have full rank.

%In the following sections, it will become
%clear that all eigenvalues %$\lambda\in\Sigma^{\varepsilon}\setminus\mathcal{B}_{r}(\{a_{0,1},a_{0,4}\})$
%of Eq.~(\ref{eq:cheq-ex}), for some sufficiently small $r>0$, approach
%minus infinity as $a_{0,3}\to0.$ Here, $\mathcal{B}_{r}(X)=\bigcup_{x\in X}\left\{ %z\in\mathbb{C}|\,\left|z-x\right|<r\right\} $
%denotes the set of balls of radius $r$ around a set $X\subset\mathbb{C}.$ 

%We remark that (ND) is satisfied if algebraic and geometric multiplicity
%of the eigenvalue zero coincide for all matrices $A_{k},\underline{k}\leq k\leq n$.
%In particular, condition (ND) is satisfied if $\ker A_{k}=\ker A_{k}^{2},$
%for all $\underline{k}\leq k\leq n.$ The following result can be
%interpreted as follows: If the leading order matrices $A_{k},\underline{k}\leq k\leq n$,
%more precisely $A_{k,4}^{(k+1)},\underline{k}\leq k<n$, are not invertible,
%there are certain spectral subsets of $\Sigma^{\varepsilon}$, which
%can be hierarchically approximated by the truncated spectra $ %$\tilde{\Sigma}_{k}^{\varepsilon},\underline{k}\leq k\leq n$. 
\begin{theorem}\label{thm:degeneracy-spec} 
Let $\det$$A_{n}=0$, $k$ be such that $\underline{k}-1\leq k\leq n-1$, and (ND) be satisfied. Further, let $\varepsilon>0$ be sufficiently small and $\mu_\varepsilon\in\tilde{\Sigma}_{k}^{\varepsilon}$ ($\mu_\varepsilon\in\tilde{S}_{0}^-$ for k=0). Then there exits a small neighborhood $U^\varepsilon(\mu_\varepsilon)\subset\mathbb{C}$ of $\mu_\varepsilon$ such that the number of eigenvalues $\lambda_\varepsilon\in\Sigma^\varepsilon\cap U^\varepsilon(\mu_\varepsilon)$ equals the multiplicity of $\mu_\varepsilon$ as a zero of $\tilde{\chi}^\varepsilon_k$.

%For every $\varepsilon>0$ and $\mu_\varepsilon\in\tilde{\Sigma}_{k}^{\varepsilon}$, there is a $\delta$ such that Then there exists $\delta_k$, such that for all $0<\delta<\delta_k$ there exists $ and $\varepsilon_k$ such that $|\mu_\varepsilon-\lambda_\varepsilon|<\delta\varepsilon^k$ for all 
\end{theorem}

The following Sec.~\ref{subsec:hierarchical-splitting} shows that eigenvalues with positive real part can be approximated in a similar way. 

\subsection{Hierarchical splitting and asymptotic spectrum\label{subsec:hierarchical-splitting} }

Consider the case when there exists an eigenvalue $\lambda^\varepsilon$ with a positive real part for an arbitrary small $\varepsilon$. It is easy to see that
\begin{equation}
\left\Vert -\lambda^{\varepsilon}I+A_{0}\right\Vert \leq\sum_{k=1}^{n}\left\Vert A_{k}\right\Vert \exp\left(-\Re(\lambda^{\varepsilon})\sigma_{k}\varepsilon^{-k}\right)\label{eq:pert-0}
\end{equation}
where $\left\Vert \cdot\right\Vert $ is an induced matrix norm. If the real part of $\lambda^\varepsilon$ is uniformly bounded from zero, we have $\sum_{k=1}^{n}\left\Vert A_{k}\right\Vert \exp\left(-\Re(\lambda)\varepsilon^{-k}\right)\to0$
as $\varepsilon\to0$. Therefore, the limiting solution $\lambda^{0}=\lim_{\varepsilon\to0}\lambda^{\varepsilon}$
is an eigenvalue of $A_{0}$ with positive real part (if it exists).
This suggests that part of the spectrum (the so-called strong unstable
spectrum, see Definition~\ref{def:strong}) with this specific scaling
property can be approximated by eigenvalues of $A_{0}$ with positive
real part, and we expect an error which is exponentially small in
$\varepsilon$ as $\varepsilon\to0$ (Theorem \ref{thm:spec-approx}). 
\begin{definition}
\label{def:strong}Let 
\[
S_{0}:=\left\{ \lambda\in\mathbb{C}\,|\,\,\det\left[-\lambda I+A_{0}\right]=0\right\} .
\]
The set
\begin{equation}
S_{0}^{+}:=S_{0}\cap\left\{ \lambda\in\mathbb{C}\,|\,\Re(\lambda)>0\right\} ,\label{eq:asympt-strong-spec}
\end{equation}
is called the \textit{asymptotic strong unstable spectrum} and the set 

\begin{equation}\label{eq:asympt-strong-spec-1}
\mathcal{A}_{0}:=S_{0}^{+}\cup\tilde{S}_{0}^{-}
\end{equation}
is called the \textit{asymptotic strong spectrum}. Let $\mathcal{B}_{r}(X)=\bigcup_{x\in X}\left\{ z\in\mathbb{C}|\,\left|z-x\right|<r\right\} $
denote the set of balls around a set $X\subset\mathbb{C}.$ Let $r_{0}:=\min\left\{ \left|\lambda-\mu\right|\ \negmedspace,\lambda,\mu\in S_{0},\lambda\ne\mu\right\} $
and $$r:=\frac{1}{3}\min\left\{ r_{0},\mbox{dist}(S_{0},i\mathbb{R})\right\} ,$$
then the sets 
\begin{equation}
\Sigma_{su}^{\varepsilon}:=\Sigma^{\varepsilon}\cap\mathcal{B}_{r}(S_{0}^{+}),\quad\Sigma_{ss}^{\varepsilon}:=\Sigma^{\varepsilon}\cap\mathcal{B}_{r}(\tilde{S}_{0}^{-}),\quad\Sigma_{s}^{\varepsilon}:=\Sigma^{\varepsilon}\cap\mathcal{B}_{r}(\mathcal{A}_{0})\label{eq:strong-spec}
\end{equation}
are called \textit{strong unstable spectrum}, \textit{strong stable spectrum}
and \textit{strong spectrum}, respectively. The set 
\begin{equation}\label{eq:pseudocont-spectrum}
\Sigma_{c}^{\varepsilon}:=\Sigma^{\varepsilon}\setminus\Sigma_{s}^{\varepsilon}
\end{equation}
is called the \textit{pseudo-continuous spectrum.}
\end{definition}
$S_{0}$ can be obtained by formal truncation of the characteristic
equation after $A_{0}$, i.e. neglecting the terms including delays. Note that Sec.~\ref{subsec:Low-Rank-perturbations} provides conditions under which is possible that also specific eigenvalues with negative real part
can be approximated by eigenvalues of $A_{0}$ (see Theorem~\ref{thm:spec-approx}.\ref{thm:spec-approx-stable-singular}).
Analogously, one defines the following truncated expressions of higher
order: Similar to our observation above, we have a splitting of the
spectral subsets with respect to the different time scales corresponding
to the hierarchy of delays. Consider an eigenvalue with real part $\Re(\lambda) = \gamma \varepsilon^k$ asymptotically as $\varepsilon\to0$ and $\gamma >0 $.
Then $\Delta^{\varepsilon}(\lambda)$ has the leading order representation
\[
-i\Im(\lambda)I+A_{0}+\sum_{j=1}^{k-1}A_{j}\exp\left(-i\sigma_{j}\varepsilon^{k-j}\Im(\lambda)\right)+A_{k}\exp\left(-\sigma_{k}\gamma-i\sigma_{k}\Im(\lambda))\right)
\]
as $\varepsilon\to 0$. This observation motivates the following definitions. 
\begin{definition}
\label{def:trunc-spec}Define the functions $\chi_{1}:\mathbb{R\times C}\to\mathbb{C}$,
$\chi_{k}:\mathbb{R\times S}^{k-1}\times\mathbb{C}\to\mathbb{C}$,
$k=2,\dots,n,$ 
\begin{eqnarray}
\chi_{1}\left(\omega;Y\right) & := & \det\left(-i\omega I+A_{0}+A_{1}Y\right),
\label{eq:ce-t1}\\
\chi_{k}\left(\omega,\varphi_{1},\dots,\varphi_{k-1};Y\right) & := & \det\left(-i\omega I+A_{0}+\sum_{j=1}^{k-1}A_{j}e^{-i\sigma_j\varphi_{j}}+A_{k}Y\right),
\label{eq:ce-tk}
\end{eqnarray}
and the corresponding asymptotic spectra
\begin{eqnarray*}
S_{1} & := & \left\{ \gamma+i\omega\in\mathbb{C}\,|\,\exists\psi\in\mathbb{R}:\,\chi_{1}\left(\omega,e^{-\sigma_{1}\gamma-i\psi}\right)=0\right\} ,\\
S_{k} & := & \left\{ \gamma+i\omega\in\mathbb{C}\,|\,\exists\psi,\varphi_{1},\dots,\varphi_{k-1}\in\mathbb{R}:\,\chi_{k}\left(\omega,\varphi_{1},\dots,\varphi_{k-1},e^{-\sigma_{k}\gamma-i\psi}\right)=0\right\} .
\end{eqnarray*}
The sets
\begin{equation}\label{eq:asymptotic-continuous-unstable-spec}
S_{k}^{+}:=S_{k}\bigcap\left\{ \lambda\in\mathbb{C}\,|\,\Re(\lambda)>0\right\} 
\end{equation}
are called \textit{the asymptotic continuous unstable $\tau_{k}$-spectrum}
for all $k=1,\dots,n,$ respectively. $S_{n}$ is called the asymptotic
continuous $\tau_{n}$-spectrum. 

If $\det A_{n}=0$, additionally
define $\tilde{\chi}_{k}:\mathbb{R\times S}^{k-1}\times\mathbb{C}\to\mathbb{C}$,
$k=\underline{k},\dots,n-1$ 
\[
\tilde{\chi}_{k}\left(\omega,\varphi_{1},\dots,\varphi_{k-1};Y\right):=\det\left[-i\omega J_{1}^{(k+1)}+A_{0,1}^{(k+1)}+\sum_{j=1}^{k-1}A_{k,1}^{(k+1)}e^{-i\sigma_j\varphi_{j}}+A_{k,1}^{(k+1)}Y\right],
\]
and if $\underline{k}=1$, $\tilde{\chi}_{1}:\mathbb{R\times C}\to\mathbb{C}$,
\begin{eqnarray*}
\tilde{\chi}_{1}\left(\omega;Y\right) & := & \det\left(-i\omega J_{1}^{(2)}+A_{0,1}^{(2)}+A_{1,1}^{(2)}Y\right),
\end{eqnarray*}
and the corresponding asymptotic continuous spectra
\begin{eqnarray*}
\tilde{S}_{1} & := & \left\{ \gamma+i\omega\in\mathbb{C}\,|\,\exists\psi\in\mathbb{R}:\,\tilde{\chi}_{1}\left(\omega,e^{-\sigma_{1}\gamma-i\psi}\right)=0\right\} ,\\
\tilde{S}_{k} & := & \left\{ \gamma+i\omega\in\mathbb{C}\,|\,\exists\psi,\varphi_{1},\dots,\varphi_{k-1}\in\mathbb{R}:\,\tilde{\chi}_{k}\left(\omega,\varphi_{1},\dots,\varphi_{k-1},e^{-\sigma_{k}\gamma-i\psi}\right)=0\right\} .
\end{eqnarray*}
The sets
\begin{equation}\label{asymptotic-continuous-stable-spec}
\tilde{S}_{k}^{-}:=\tilde{S}_{k}\bigcap\left\{ \lambda\in\mathbb{C}\,|\,\Re(\lambda)<0\right\} 
\end{equation}
are called the \textit{asymptotic continuous stable $\tau_{k}$-spectrum} for
all $k=\underline{k},\dots,n-1$ respectively. 

\begin{equation}\label{asymptotic-continuous-spectrum}
\mathcal{A}_{k}:=S_{k}^{+}\cup\tilde{S}_{k}^{-},\,1\leq k<n,\qquad\mathcal{A}_{n}:=S_{n}
\end{equation}
are called \textit{asymptotic continuous $\tau_{k}$-spectra}. Additionally, for fixed $1\leq k<n$, consider the scaling
function $\Pi_{\varepsilon}^{(k)}:\mathbb{C}\to\mathbb{C},$
\begin{equation}
\Pi_{\varepsilon}^{(k)}(a+ib):=a\varepsilon^{-k}+ib.\label{eq:scaling-f}
\end{equation}
We define the corresponding spectral subsets 
\[
\Sigma_{k,\nu}^{\varepsilon}:=\left\{ \lambda\in\Sigma^{\varepsilon}\,|\,\mbox{dist}\left(\Pi_{\varepsilon}^{(k)}(\lambda),\mathcal{A}_{k}\right)<\nu,|\Re(\lambda)|>\nu\right\} ,\quad k=1,\dots,n-1,
\]
\[
\Sigma_{n,\nu}^{\varepsilon}:=\Sigma^{\varepsilon}\setminus\bigcup_{k=0}^{n-1}\Sigma_{k,\nu}^{\varepsilon}.
\]
of the pseudo continuous spectrum.
\end{definition}
The following Theorem contains our main result. We show that as $\varepsilon\to0$
the strong spectrum $\Sigma_{s}^{\varepsilon}$ converges to the asymptotic
strong spectrum $\mathcal{A}_{0}$ and the pseudo-continuous spectrum
converges to the imaginary axis. The rescaled spectral sets
$\Pi_{\varepsilon}^{(k)}(\Sigma_{k,\nu}^{\varepsilon})$ converge
to the sets given by asymptotic continuous $\tau_{k}$-spectra $\mathcal{A}_{k}$. Recall Defs.~\ref{def:strong} and \ref{def:trunc-spec}.
\begin{theorem}[spectrum approximation]
\label{thm:spec-approx}Assume (ND). 
\begin{enumerate}
\item[(i)] \label{thm:spec-approx-strong-unst}Let $\mu\in S_{0}^{+}$. Then
for $0<\delta\le r$ there exists $\varepsilon_{0}>0$ such that for
$0<\varepsilon<\varepsilon_{0}$ the number of eigenvalues in $\Sigma^{\varepsilon}\cap\mathcal{B}_{\delta}(\mu)$
counting multiplicities equals the multiplicity of $\mu$ as an
eigenvalue of $A_{0}$. 
\item[(ii)]  \label{thm:spec-approx-stable-singular}Let $\mu\in\tilde{S}_{0}^{-}$.
Then for $0<\delta\le r$ there exists $\varepsilon_{0}>0$ such that
for $0<\varepsilon<\varepsilon_{0}$ the number of eigenvalues in
$\Sigma^{\varepsilon}\cap\mathcal{B}_{\delta}(\mu)$ counting
multiplicities equals the multiplicity of $\mu$ as a solution of $\tilde{\chi}_{0}(\mu)=0$.
\item[(iii)]  \label{thm:spec-approx-weak-unst} Let $k=1,\dots,n$ and $\chi_{k}$
be nontrivial. For $\mu\in\mathcal{A}_{k}$, and $\delta>0$ there
exists $\varepsilon_{0}>0$ such that for $0<\varepsilon<\varepsilon_{0}$
there exists $\lambda\in\Sigma_{k,\delta}^{\varepsilon}\subset\Sigma^{\varepsilon}$
such that $\left|\Pi_{\varepsilon}^{(k)}(\lambda)-\mu\right|<\delta$.
\item[(iv)]  \label{thm:spec-approx-stable}Let $R>0$. For $0<\delta>0$ there
exists $\varepsilon_{0}>0$ such that for $0<\varepsilon<\varepsilon_{0}$
and $\lambda\in\Sigma_{c}^{\varepsilon}$ with $\left|\Im\left(\lambda\right)\right|<R$,
we have $\left|\Re\left(\lambda\right)\right|<\delta$ and there exists
$1\le k\le n$ and $\mu\in\mathcal{A}_{k}$ such that $\left|\Pi_{\varepsilon}^{(k)}(\lambda)-\mu\right|<\delta$.
\end{enumerate}
\end{theorem}
Theorem~\ref{thm:spec-approx} has several implications for the
stability of Eq.~(\ref{eq:DDE}) for sufficiently large values of
the delays. In particular, if the asymptotic unstable spectra $S_{k}^{+}$
are empty for all $0\leq k\leq n$, then Eq.~(\ref{eq:DDE}) is asymptotically
stable. By construction, we have $\mathcal{A}_{k}\subset S_{k}$. As a result, we
can explore the structure of the sets $S_{k},$ without knowing $\tilde{S}_{k}^{-}$
explicitly, but keep in mind that there are $\mu\in S_{k}$ with $\mu\notin\mathcal{A}_k$.

In Sec.~\ref{subsec:Spectral-manifolds}, we provide explicit formulas
for the sets $S_{k}$ (and therefore $\mathcal{A}_{k}$) and introduce the concept
of a spectral manifold. The presented results will clarify the structure
of the asymptotic spectrum. 

\subsection{Spectral manifolds\label{subsec:Spectral-manifolds} }
We introduce the notion of spectral manifolds as solutions to
\begin{equation}
\chi_{k}\left(\omega,\varphi_{1},\dots,\varphi_{k-1};Y\right)=0,\quad1\leq k\leq n\label{eq:cheq-k}
\end{equation}
(compare $S_{k}$ in Definition~\ref{def:trunc-spec}). Equation
(\ref{eq:cheq-k}) can be thought of as a polynomial in $Y$ of degree
$d_{k}=\text{rank}A_{k}$. To start with, let us assume $d_{k}=d$.
For fixed $(\omega,\varphi_{1},\dots,\varphi_{k-1})\in\mathbb{R\times S}^{k-1},$
this equation has $d$ complex roots. Thus, there exist $d$ continuous
functions $Y_{l}^{(k)}:\,\mathbb{R\times S}^{k-1}\to\mathbb{C}$
\[
\chi_{k}\left(\omega,\varphi_{1},\dots,\varphi_{k-1};Y_{l}^{(k)}(\omega,\varphi_{1},\dots,\varphi_{k-1})\right)=0
\]
 for $1\le l\le d$. One defines 
\[
\gamma_{l}^{(k)}(\omega,\varphi_{1},\dots,\varphi_{k-1}):=-\frac{1}{\sigma_{k}}\ln\left|Y_{l}^{(k)}(\omega,\varphi_{1},\dots,\varphi_{k-1})\right|
\]
and extend it continuously onto $\mathbb{R}$ with values in $\mathbb{R}\cup\{-\infty,\infty\}$.
The functions $\gamma_{1}^{(k)},\dots,\gamma_{d_{k}}^{(k)}$ are called
spectral manifolds of $S_{k}$. They can be obtained from straightforward
computation, analytically in many cases. If $d_{k}<d$, that is $A_{k}$
has not full rank, spectral manifolds can become locally or globally
degenerate; they might seize to exits for certain parameter values.
In this case, denote $U_{k,1},V_{k,1}\in\mathbb{C}^{d\times(d-d_{k})}$
the matrices containing the left and right singular vectors of $A_{k}$
corresponding to the singular value zero.
\begin{theorem}[Spectral manifolds]
\label{thm:spec-mnf}Assume (ND) and let $1\le k\le n$ be fixed
with $d_{k}:=\text{rank}\,A_{k}$. Then, 
\begin{enumerate}
\item[(i)] There exist $d_{k}$ continuous functions $\gamma_{1}^{(k)},\dots,\gamma_{d_{k}}^{(k)}:\mathbb{R}^{k}\to\mathbb{R}\cup\{-\infty,\infty\}$
such that 
\begin{eqnarray*}
S_{k} & = & \bigcup_{l=1}^{d_{k}}\left\{ \gamma_{l}^{(k)}(\omega,\varphi_{1},\dots,\varphi_{k-1})+i\omega,\ (\omega,\varphi_{1},\dots,\varphi_{k-1})\in\mathbb{R\times S}^{k-1},\right.\\
 &  & \left.\qquad\,\gamma_{l}^{(k)}(\omega,\varphi_{1},\dots,\varphi_{k-1})\notin\{-\infty,\infty\}\right\} .
\end{eqnarray*}
\item[(ii)]\label{thm:spec-mnf-singularities+}If $d_{k}<d$, for any
$(\omega,\varphi_{1},\dots,\varphi_{k-1})\in\mathbb{R\times S}^{k-1}$
with
\begin{equation}
\det\left(U_{k,1}^{\ast}\left[-i\omega I+A_{0}+\sum_{j=1}^{k-1}A_{j}e^{-i\sigma_j\varphi_{j}}\right]V_{k,1}\right)\neq0,\label{eq:sing1}
\end{equation}
there exists $l\in\{1,\dots,d_{k}\}$ such that the following holds
true:\linebreak{}
 $\gamma_{l}^{(k)}(\omega,\varphi_{1},\dots,\varphi_{k-1})=\infty$,
if and only if 
\begin{equation}
\det\left(-i\omega I+A_{0}+\sum_{j=1}^{k-1}A_{j}e^{-i\sigma_j\varphi_{j}}\right)=0.\label{eq:sing2}
\end{equation}
If $k<n$, the set of zero points of the spectral manifold $\gamma_{l}^{(k)}$
coincides with the set of singular points of the spectral manifold
$\gamma_{m}^{(k+1)}$ for some $m\in\{1,\dots,d_{k+1}\}$.
\item[(iii)]\label{thm:spec-mnf-singularities-}If $d_{k}<d$, for any
$(\omega,\varphi_{1},\dots,\varphi_{k-1})\in\mathbb{R\times S}^{k-1}$
with
\[
\det\left(-i\omega I+A_{0}+\sum_{j=1}^{k-1}A_{j}e^{-i\sigma_j\varphi_{j}}\right)\ne0,
\]
there exists $l\in\{1,\dots,d_{k}\}$ such that the following holds
true:\linebreak{}
 $\gamma_{l}^{(k)}(\omega,\varphi_{1},\dots,\varphi_{k-1})=-\infty$,
if and only if
\[
\det\left(U_{k,1}^{\ast}\left[-i\omega I+A_{0}+\sum_{j=1}^{k-1}A_{j}e^{-i\sigma_j\varphi_{j}}\right]V_{k,1}\right)=0.
\]
\end{enumerate}
\end{theorem}
Generically, the set of zero points of the spectral manifold $\gamma_{l}^{(k)}(\omega,\varphi_{1},\dots,\varphi_{k-1})$
is locally a $k-1$ dimensional manifold, and a set of singular points
is locally a $k-2$ dimensional manifold. The case $k=1$ is studied
in \cite{Lichtner2011}, and it is shown that the singularity of a
spectral curve $\gamma^{(1)}$ can be only observed changing one additional
parameter. For the case $k>1$, the singularity of $S_{k}^{+}$ is
generically expected when the asymptotic unstable spectrum $S_{k-2}^{+}$
is nonempty. The following Corollary is an immediate consequence
of Theorems \ref{thm:degeneracy-spec},\ref{thm:spec-approx} and \ref{thm:spec-mnf}.
\begin{corollary}
\label{coro:stability}Assume (ND). (i) If all spectral manifolds
$S_{k}$, $k=1,\dots,n$ are in the negative half-plane, i.e. $\gamma_{l}^{(k)}<0$
for all $\omega,\varphi_{1},\dots,\varphi_{k-1}\in\mathbb{R}$ and
$S_{0}^{+}=\emptyset$, then there exists $\varepsilon_{0}>0$ such
that for $0<\varepsilon<\varepsilon_{0}$, $x\equiv0$ is
exponentially stable in Eq.~(\ref{eq:DDE}).\\
(ii) If some spectral manifold admits positive value, i.e. $\gamma_{l}^{(k)}>0$
for some $0<l\le k\le n$ and $\omega,\varphi_{1},\dots,\varphi_{k-1}\in\mathbb{R}$,
or $S_{0}^{+}\ne\emptyset$, then there exists $\varepsilon_{0}>0$
such that for $0<\varepsilon<\varepsilon_{0}$, $x\equiv0$ is exponentially unstable in Eq.~(\ref{eq:DDE}). 
\end{corollary}
In particular, it is evident that the onset of instability is mediated
by the crossing of an $n$-dimensional spectral manifold corresponding
to the timescale $\tau_{n}$. In order to see this, observe that $i\omega\in\Sigma^{\varepsilon}$
implies that $\chi_{n}\left(\omega,\varphi_{1},\dots,\varphi_{n-1};e^{i\varphi_{n}}\right)=0$,
where $\varphi_{k}=\sigma_{k}\omega/\varepsilon^{-k}$ for all $1\leq k\leq n$.
It is very easy to assess whether $S_{0}^{+}\ne\emptyset$, as we
only have to compute the positive eigenvalue of $A_{0}$. 

As we see from Theorems \ref{thm:spec-approx} and \ref{thm:spec-mnf},
the pseudo-continuous part of the spectrum can be understood geometrically
as a certain projection of the manifolds $\gamma_{l}^{(k)}$ to
the complex plane. The resulting projections are called here $S_{k}$.
Certain parts of these projections, called here $\mathcal{A}_{k}$
are asymptotically filled with the eigenvalues $\Sigma_{k,\delta}^{\varepsilon}$.
In the case when $\mathcal{A}_{k}$ is one-dimensional ($k=1$), this
is a projection of curves, and as a result, the asymptotic spectrum
has the form of curves - such a case was considered in details
in \cite{Lichtner2011}. We remark here, that already Bellman and Cooke
\cite[p. 399]{Bellman1963} noticed: \textit{"They [the eigenvalues of Eq.~(\ref{eq:DDE})] are thus seen to lie in a finite
number of chains [here $\Sigma_{k,\delta}^{\varepsilon}$,
$k=1,\dots,n$]. Each chain consists of a countable infinity
of zeros."}

For larger $k$, the spectrum is described by the projection of some
higher-dimen- sional manifold, and as $\varepsilon\to0$, the corresponding
sets $\Sigma_{k,\delta}^{\varepsilon}$ become densely filled with
eigenvalues. This geometric property of the spectrum provides a
motivation to refer to the sets $\mathcal{A}_{k}$ as asymptotic continuous.

\section{Example: scalar equation with two large hierarchical delays\label{sec:example}}

In order to illustrate the obtained results, we treat the scalar linear
DDE with two large hierarchical delays in more detail, and study the
set of solutions (eigenvalues) to the corresponding characteristic
equation
\begin{equation}
-\lambda+a+be^{-\lambda/\varepsilon}+ce^{-\lambda/\varepsilon^{2}}=0.\label{eq:ex-cheq}
\end{equation}
Theorem$\,\ref{thm:spec-approx}$ states that the solutions of (\ref{eq:ex-cheq})
can each be approximated by an element of one of the sets $S_{0}^{+}$
(asymptotic strong unstable spectrum), and $S_{1}^{+},$ $S_{2}$
(asymptotic continuous spectra). Note that in the scalar case Condition
(ND) reduces to $a,b,c\neq0$ such that there is no degenerate
spectrum. As an immediate consequence of the presented theory, for
large values of the delay, the stability boundary of the trivial equilibrium
is solely determined by the position of $S_{0}^{+},$ $S_{1}^{+},$
and $S_{2}$ with respect to the imaginary axis. We distinguish between
three different types of instability, each corresponding to one of
the sets $S_{0}^{+},S_{1}^{+}$ and $S_{2}$. If $S_{0}^{+}$ is not
empty, we say that the spectrum is strongly unstable and mean that
there are solutions of the original DDE, which grow on timescale $1$.
If $S_{1}^{+}$ or $S_{2}^{+}$ is not empty, we speak of a weak instability,
and mean that solutions grow on time-scale $\varepsilon^{-1}$ or
$\varepsilon^{-2}$, respectively. This scale separation cannot be
observed in linear delay equation with a single large delay \cite{Lichtner2011}.
In order to differentiate these two types, we also refer to them as
weak instability on timescale $\varepsilon^{-1}$ or $\varepsilon^{-2}$,
respectively. Let us focus on weak instabilities, and assume that
the strong unstable spectrum is absent, i.e. $S^+_0=\emptyset$. We compare the approximations
$S_{1}^{+}$ and $S_{2}$ to numerically computed eigenvalues. The
explicit formulas as well as necessary and sufficient conditions for
stability are contained in Sec.~\ref{subsec:Explicit-formulas};
the results are summarized in Table \ref{tab:Conditions-for-stability}.
\begin{table}
\begin{centering}
\begin{tabular*}{\textwidth}{l@{\extracolsep{\fill}}llll}
\multicolumn{3}{c}{relevant asymptotic spectra} & parameters\tabularnewline
\midrule 
\multirow{2}{*}{$\begin{array}{c}
\text{asymptotic strong}\\
\text{unstable spectrum}
\end{array}$} & \multirow{2}{*}{$S_{0}^{+}$ } & present (unstable) & $\Re(a)>0$\tabularnewline
\cmidrule{3-4} 
 &  & not present  & $\Re(a)<0$\tabularnewline
\midrule
\multirow{6}{*}{$\begin{array}{c}
\text{asymptotic continuous}\\
\text{spectrum}
\end{array}$ } & \multirow{3}{*}{$S_{1}^{+}$} & present (unstable) & $\left|b\right|>\left|\Re(a)\right|$\tabularnewline
\cmidrule{3-4} 
 &  & not present & $\left|b\right|<\left|\Re(a)\right|$\tabularnewline
\cmidrule{3-4} 
 &  & singular points & $\Re(a)=0$\tabularnewline
\cmidrule{2-4} 
 & \multirow{3}{*}{$S_{2}$} & unstable & $|c|>\left|\Re(a)\right|-|b|$\tabularnewline
\cmidrule{3-4} 
 &  & stable & $|c|<\left|\Re(a)\right|-|b|$\tabularnewline
\cmidrule{3-4} 
 &  & singular points & $\left|b\right|\geq\left|\Re(a)\right|$\tabularnewline
\midrule
\end{tabular*}\linebreak{}
\par\end{centering}
\caption{\label{tab:Conditions-for-stability}Summary of spectra and conditions
for stability of Eq.~(\ref{eq:ex-cheq}).}
\end{table}
We discuss the destabilization scenario as eigenvalues of the pseudo-continous
spectrum cross the imaginary axis. Let us fix parameters corresponding
to an exponentially stable equilibrium, i.e. $\Re(a)<0$ (no strong unstable
spectrum), $|b|<\left|\Re(a)\right|$ ($S_{1}^{+}=\emptyset$), and $|c|<\left|\Re(a)\right|-|b|$ ($S_{2}^{+}=\emptyset$). Note
that these conditions are not independent of one another: $|b|>\left|\Re(a)\right|$
($S_{1}^{+}\neq\emptyset$) implies $|c|>0>\left|\Re(a)\right|-|b|$ ($S_{2}\neq\emptyset$). 

Following Table~\ref{tab:Conditions-for-stability}, $S_{2}$ crosses
the imaginary axis if $|c|$ is increased beyond the threshold value
$\left|\Re(a)\right|-|b|,$ see Fig. \ref{fig:spec-1}. Perturbations
in the neighborhood of the equilibrium grow on the timescale $\varepsilon^{-2}$
and the equilibrium loses stability. Leaving $c$ unchanged, we vary
$b$ such that $|b|>\left|\Re(a)\right|$ and $S_{2}$ develops a
singularity, see Fig. \ref{fig:spec-2}. Simultaneously, $S_{1}$
crosses the imaginary axis and $S_{1}^{+}\neq\emptyset$, see Fig.~\ref{fig:spec-2}.
Here, perturbations in the neighborhood of the equilibrium grow on
the timescale $\tau_{1}=\varepsilon^{-1}$ and the equilibrium has
become qualitatively more unstable.
\begin{figure}
\begin{centering}
\begin{minipage}[t]{0.32\linewidth}%
\begin{center}
\includegraphics[width=1\linewidth]{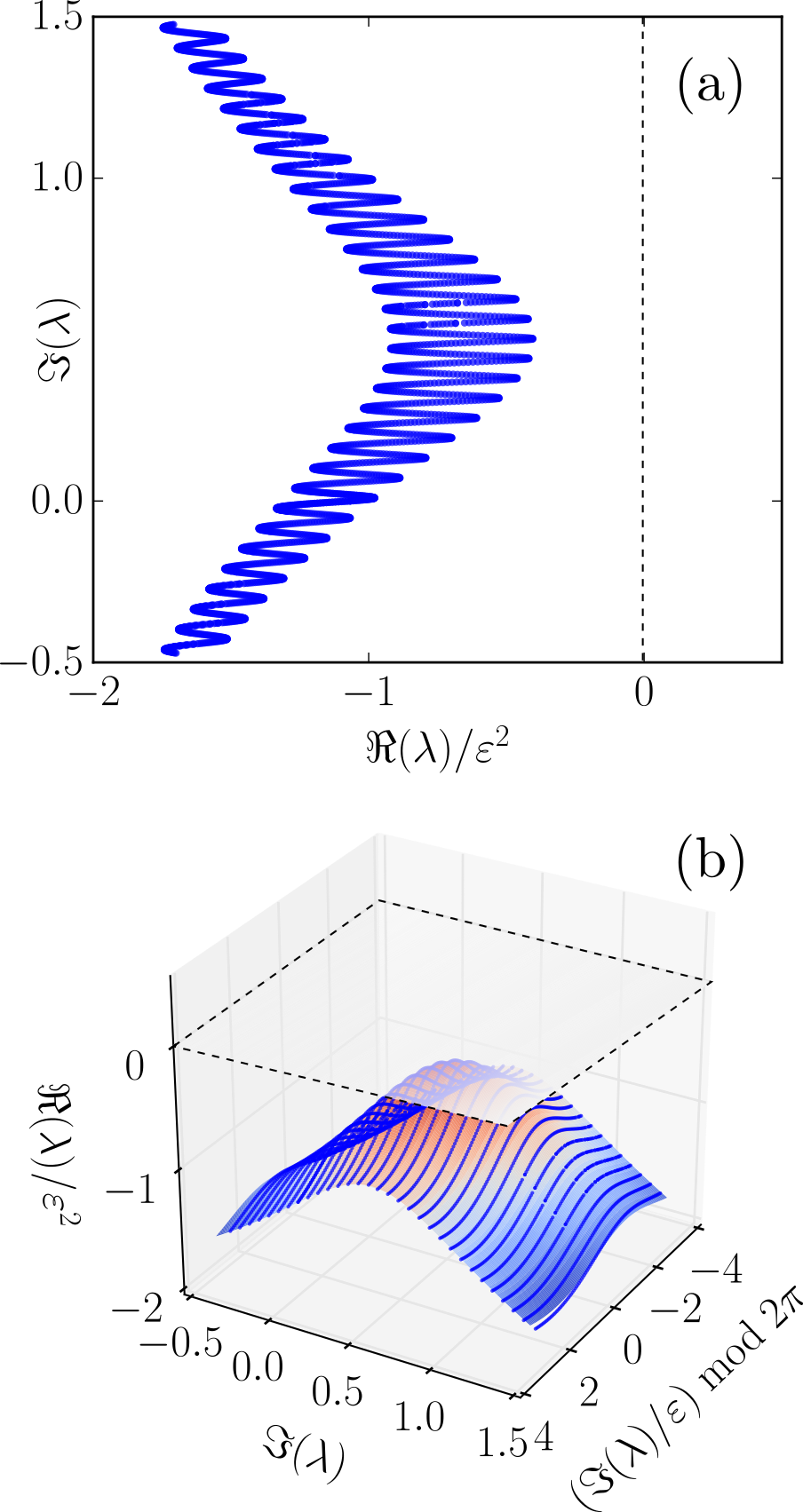}
\par\end{center}%
\end{minipage}\hfill{}%
\begin{minipage}[t]{0.32\linewidth}%
\begin{center}
\includegraphics[width=1\linewidth]{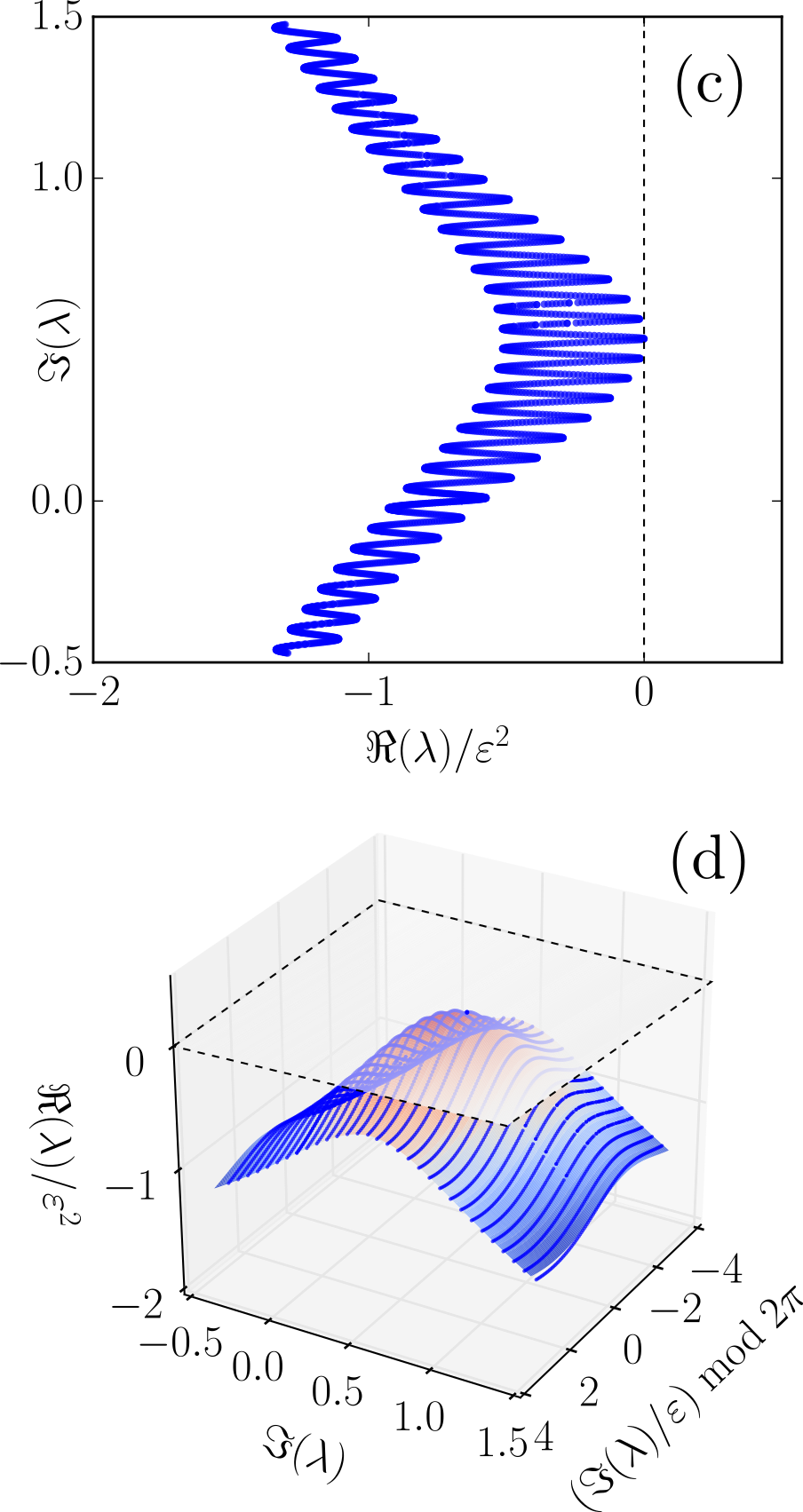}
\par\end{center}%
\end{minipage}\hfill{}%
\begin{minipage}[t]{0.32\linewidth}%
\begin{center}
\includegraphics[width=1\linewidth]{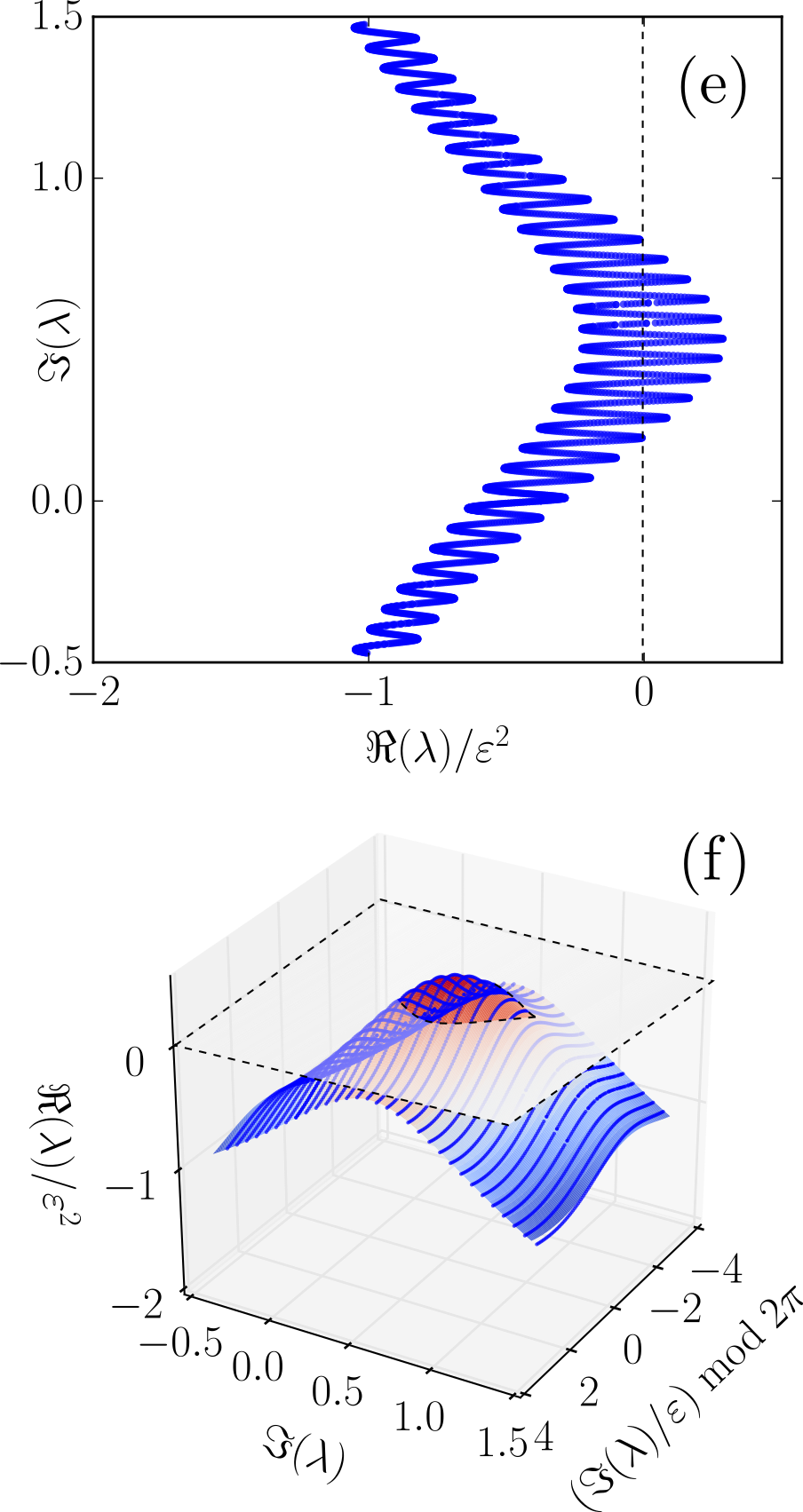}
\par\end{center}%
\end{minipage}
\par\end{centering}
\caption{\label{fig:spec-1}Eigenvalues of the characteristic equation (\ref{eq:ex-cheq})
corresponding to two hierarchical delays. Panels (a)-(f) show the
destabilization of the spectrum varying parameter $c$ (columns from
left to right: $c=0.2$ (stable), $c=0.3$ (neutral), $c=0.4$ (unstable)).
Panels (a),(c),(e) show the spectrum (real part rescaled). Panels
(b),(d),(f): approximation of the spectrum via the two-dimensional
spectral manifold $\gamma^{(2)}$ ($S_{2},$ colored surface). Other
parameters are $a=-0.4+0.5i$, $b=0.1$, and $\varepsilon=0.01$.
$S_{0}^{+}$ and $S_{1}^{+}$ are not present. Blue dots are numerically
computed eigenvalues.}
\end{figure}
In Fig.~\ref{fig:spec-2}, one observes the hierarchical splitting
of the spectrum in terms of the sets $S_{1}^{+}$ and $S_{2}$. This
phenomenon is not observed in systems with single large delay \cite{Lichtner2011}.

We remark that the above mentioned destabilization governed by the
characteristic equation (\ref{eq:ex-cheq}) was observed numerically in \cite{Yanchuk2014a,Yanchuk2015}.
It was shown that such an instability, accompanied by an appropriate
nonlinear saturation, can lead to a formation of spiral-wave like
dynamics. Below, we derive $S_0^+,S_1^+$ and $S_2$ explicitely.

\begin{figure}
\begin{centering}
\begin{minipage}[t]{0.32\linewidth}%
\begin{center}
\includegraphics[width=1\linewidth]{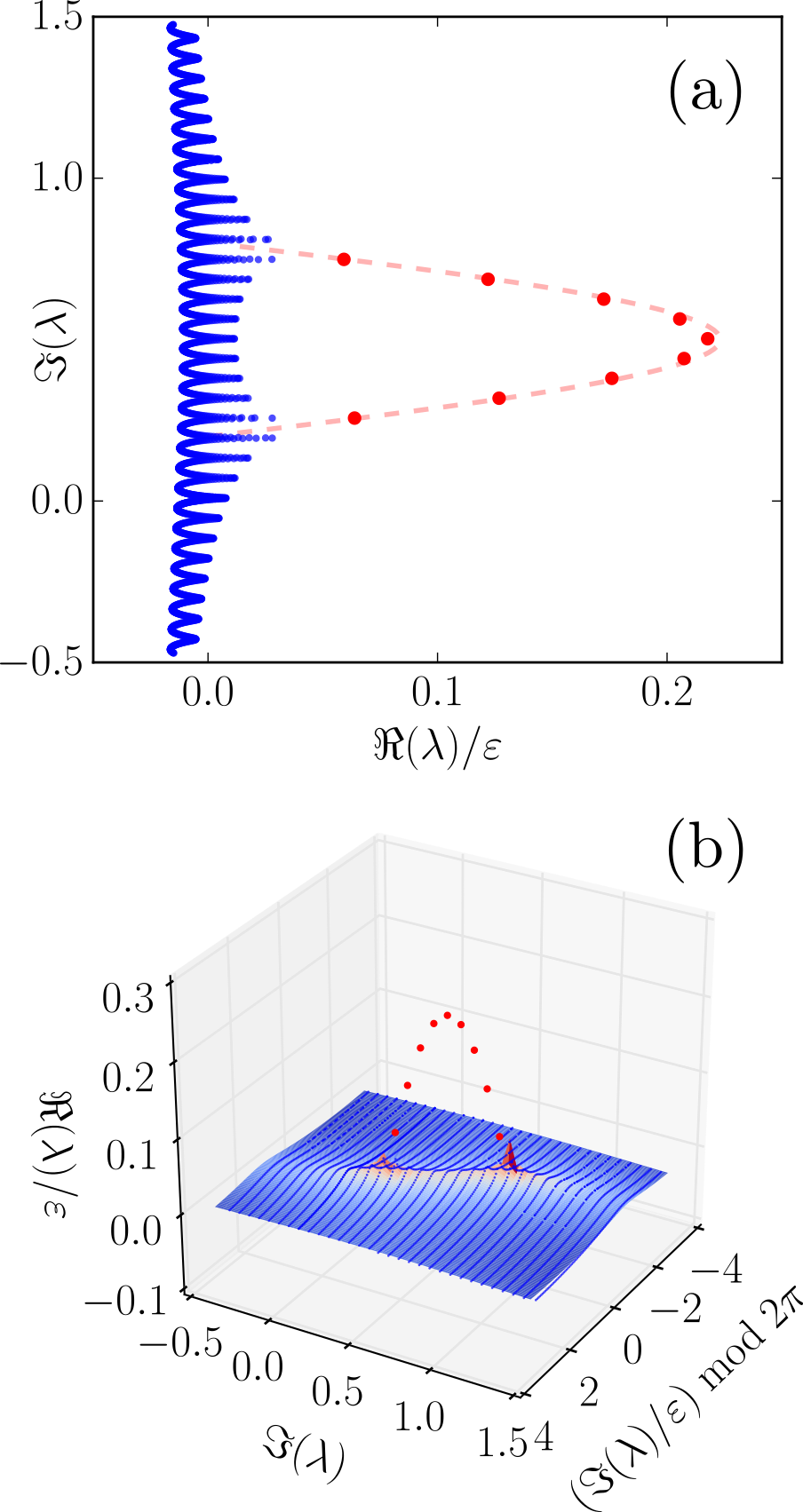}
\par\end{center}%
\end{minipage}\hfill{}%
\begin{minipage}[t]{0.32\linewidth}%
\begin{center}
\includegraphics[width=1\linewidth]{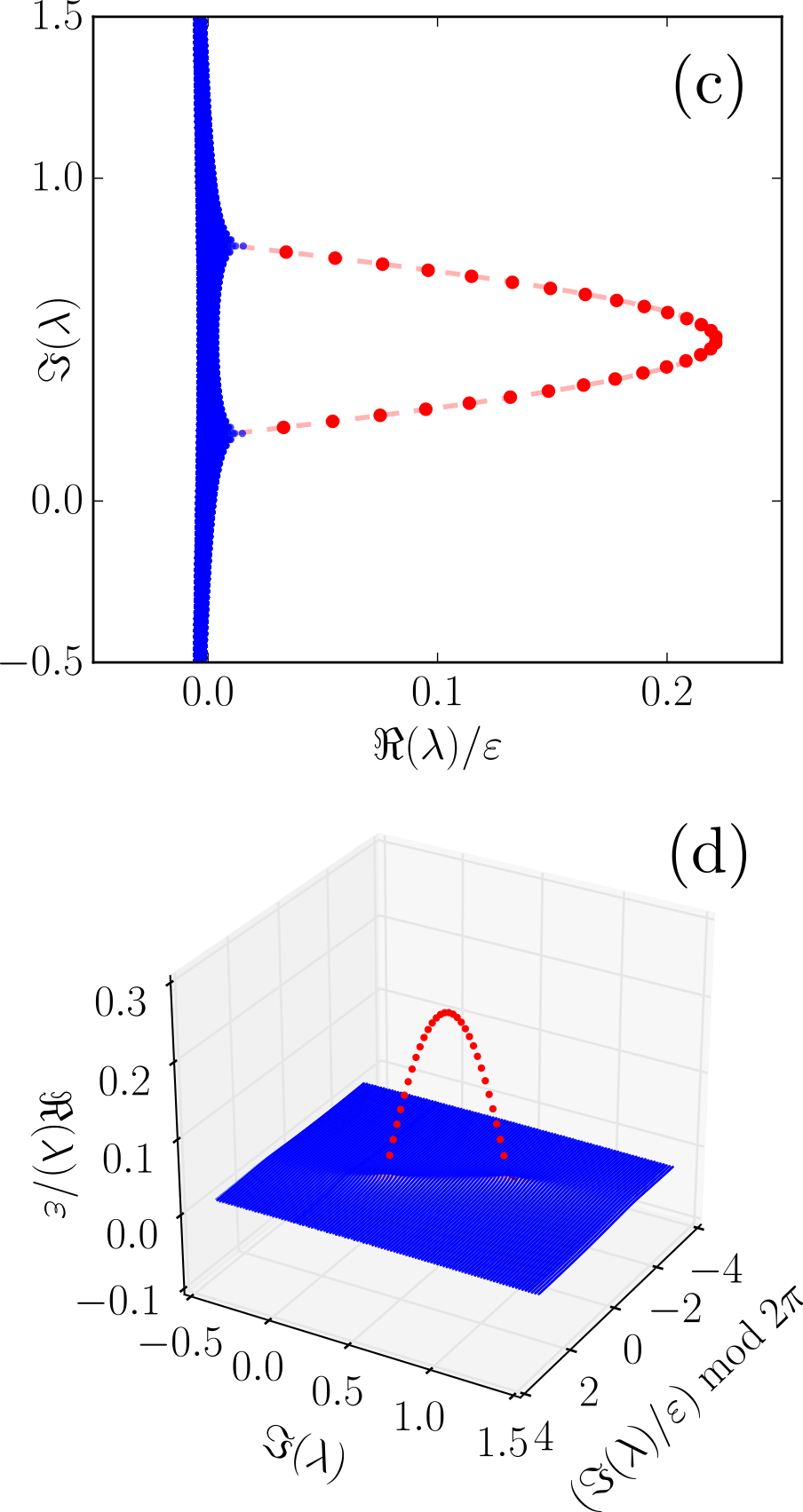}
\par\end{center}%
\end{minipage}\hfill{}%
\begin{minipage}[t]{0.32\linewidth}%
\begin{center}
\includegraphics[width=1\linewidth]{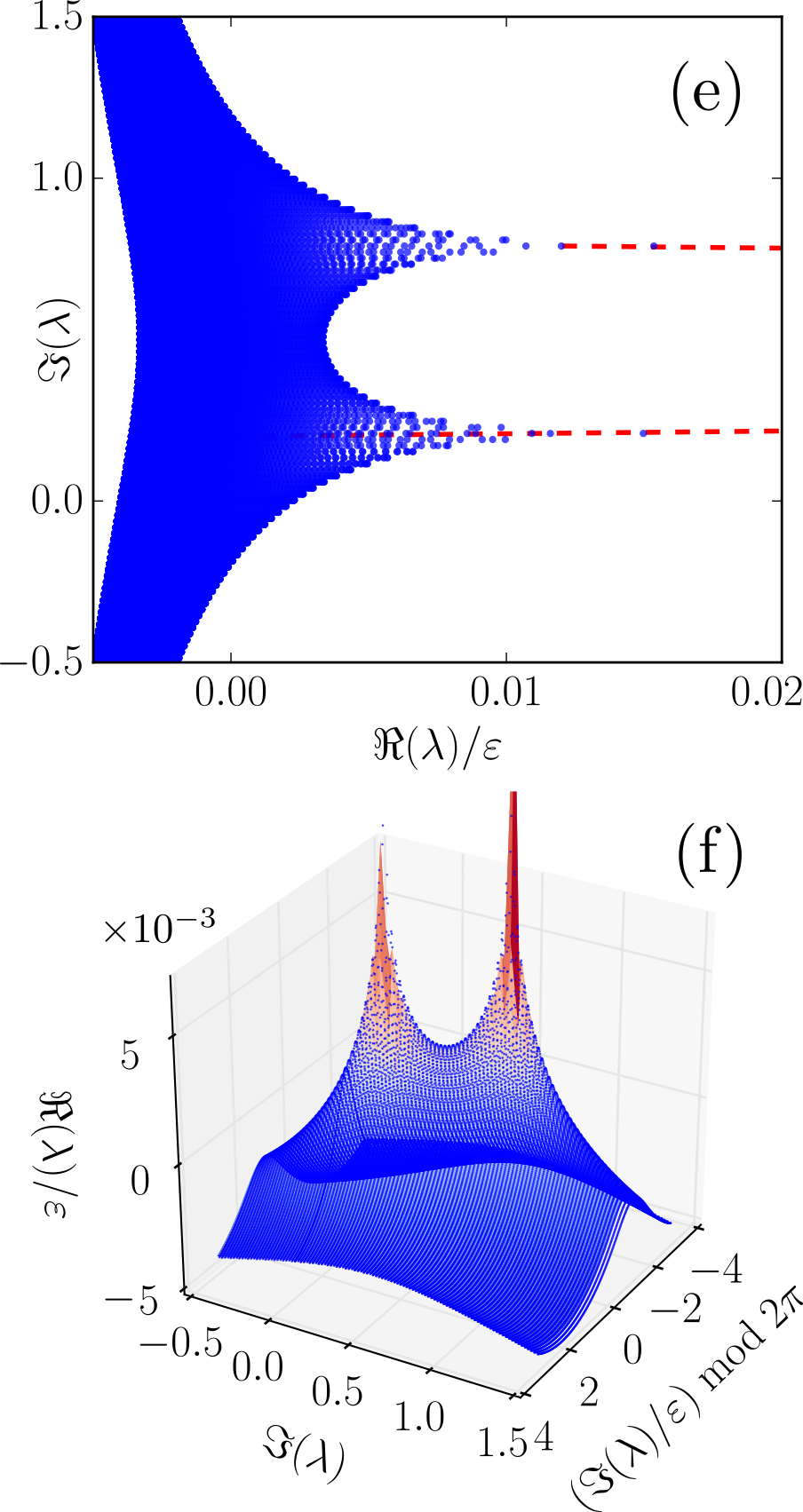}
\par\end{center}%
\end{minipage}
\par\end{centering}
\caption{\label{fig:spec-2}Eigenvalues of the characteristic equation (\ref{eq:ex-cheq})
corresponding to two hierarchical delays. Two types of spectra coexisting:
$S_{1}^{+}$ (red) and $S_{2}$ (blue). Panels (a)-(f) show the spectrum
varying parameter $\varepsilon$ (columns from left to right: $\varepsilon=0.01$,
$\varepsilon=0.003$, $\varepsilon=0.003$ (zoom)). Panels (a),(c),(e):
approximation of the $\tau_{1}-$spectrum (red) via spectral manifold
$\gamma^{(1)}$ (magenta dotted). Panels (b),(d),(f): approximation
of the $\tau_{2}-$spectrum (blue) via two-dimensional spectral manifolds
$\gamma^{(2)}$ (colored surface). Other parameters are $a=-0.4+0.5i$,
$b=0.5$, and $c=0.3$. $S_{0}^{+}$ is not present. Blue dots are
numerically computed eigenvalues.}
\end{figure}

\subsection*{Explicit formulas of asymptotic spectral manifolds\label{subsec:Explicit-formulas}}

The asymptotic strong unstable spectrum can be read off directly from
Eq.~(\ref{eq:ex-cheq}), 
\begin{equation}
S_{0}^{+}=\begin{cases}
\{a\}, & \Re(a)>0,\\
\emptyset, & \text{otherwise.}
\end{cases}\label{eq:ex-asympt-strong}
\end{equation}
As Eq.~(\ref{eq:ex-cheq}) is scalar, $\text{rank}\,b=\text{rank}\,c=1$,
the asymptotic continuous spectrum is determined by two spectral manifolds
$\gamma^{(1)},\gamma^{(2)}$ such that $S_{1}^{+}=\{\gamma^{(1)}(\omega)+i\omega,\,\omega\in\mathbb{R}\}\cap\{\lambda\in\mathbb{C}|\Re(\lambda)>0\},$
and $S_{2}=\{\gamma^{(2)}(\omega,\varphi_{1})+i\omega,\,\omega,\varphi_{1}\in\mathbb{R}\}$
respectively. These manifolds $\gamma^{(1)},\gamma^{(2)}$ can be
computed from
\[
\chi_{1}\left(\omega;Y\right)=\lambda-a-bY=0,\qquad|Y|=e^{-\gamma^{(1)}},
\]
and
\[
\chi_{2}\left(\omega,\varphi_{1};Z\right)=\lambda-a-be^{-i\varphi_{1}}-cZ=0,\qquad|Z|=e^{-\gamma^{(2)}},
\]
see Sec.~\ref{subsec:Spectral-manifolds} for details. We proceed
with the formal analysis of these manifolds. 

It follows from straightforward computation that
\begin{equation}
\gamma^{(1)}(\omega)=-\frac{1}{2}\ln\frac{(\omega-\Im(a))^{2}+\Re(a)^{2}}{|b|^{2}}.\label{eq:ex-gamma-1}
\end{equation}
$\gamma^{(1)}$ attains its global maximum at $\omega=\Im(a)$, and
$\gamma_{\max}^{(1)}=\gamma^{(1)}(\Im(a))=\ln\frac{|b|}{\left|\Re(a)\right|}$.
As a consequence, $\gamma_{\max}^{(1)}>0$ if and only if $|b|>|\Re(a)|$.
The unstable part of $S_{1}$ is then given by 
\[
S_{1}^{+}=\left\{ \gamma^{(1)}(\omega)+i\omega\,,\,\omega_{1}<\omega<\omega_{2}\right\} ,
\]
where $\omega_{1,2}=\Im(a)\pm\sqrt{|b|^{2}-\Re(a)^{2}}$ are the zero
points of $\gamma^{(1)}$. If the asymptotic strong spectrum
is neutral ($\Re(a)=0$), then $\gamma_{\max}^{(1)}$ is singular
$\gamma^{(1)}(\Im(a))=\infty.$

Similarly, $\gamma^{(2)}$ can be expressed as
\begin{eqnarray}
\gamma^{(2)}(\omega,\varphi_{1}) & = & -\frac{1}{2}\ln\frac{1}{|c|^{2}}\left[\left(\Re\left(a\right)+\left|b\right|\cos(\varphi_{1}-\mbox{Arg}(b))\right)^{2}\right.\label{eq:ex-gamma-2}\\
 &  & \left.+\left(\omega-\Im(a)+\left|b\right|\sin(\varphi_{1}-\mbox{Arg}(b))\right)^{2}\right].\nonumber 
\end{eqnarray}
Let $\varphi_{1}$ be fixed and assume $\left|b\right|<\left|\Re\left(a\right)\right|$
($S_1^+=\emptyset$), then $\gamma^{(2)}(\omega,\varphi_{1})$
defined in (\ref{eq:ex-gamma-2}) attains its global maximum 
\[
\gamma^{(2)}(\omega_{\max}(\varphi_{1}),\varphi_{1})=-\ln\frac{\left|\Re\left(a\right)+\left|b\right|\cos(\varphi_{1}-\mbox{Arg}(b))\right|}{|c|}
\]
 at $\omega=\omega_{\max}(\varphi_{1})=\Im(a)-\left|b\right|\sin(\varphi_{1}-\mbox{Arg}(b))$
and the maximum is given by $$\max_{\varphi_{1}}\gamma^{(2)}(\omega_{\max}(\varphi_{1}),\varphi_{1})=-\ln\frac{\left|\Re\left(a\right)-\left|b\right|\right|}{|c|}.$$
If $\left|b\right|\geq\left|\Re\left(a\right)\right|$, $\gamma^{(2)}$
is unbounded and the zeros of $\gamma^{(1)}$ (not necessarily isolated)
correspond to the singularities of $\gamma^{(2)}$. The corresponding
values of $\varphi_{1}$,
\[
\varphi_{\pm}=-\mbox{Arg}(b)\pm\arctan\left(\frac{\sqrt{|b|^{2}-\Re(a)^{2}}}{\Re\left(a\right)}\right).
\]
can be found from the ansatz $\chi_{2}(\omega_{1,2},\varphi_{\pm};0)=0$. In summary,
\[
\sup_{\omega,\varphi_{1}\in\mathbb{R}}\gamma^{(2)}(\omega,\varphi_{1})=\begin{cases}
-\ln\frac{\left|\Re\left(a\right)\right|-\left|b\right|}{|c|}, & \mbox{if}\,\left|\Re\left(a\right)\right|>\left|b\right|,\\
\infty & \mbox{if}\,\left|\Re\left(a\right)\right|<\left|b\right|.
\end{cases}
\]

\section{Proof of Theorems \ref{thm:degeneracy-spec}, \ref{thm:spec-approx} and \ref{thm:spec-mnf}}\label{sec:proofs}

In this section, we prove our main results stated in Sec.~\ref{sec:overview}.

\subsection*{Proof of Theorem \ref{thm:degeneracy-spec}}

We proof by induction starting from the highest order $k=n$.  
We assume that $\det A_n=0$, $A_n\neq 0$ and consider a $\mu_\varepsilon$ such that $\tilde \chi^\varepsilon _{n-1} (\mu_\varepsilon) = 0$, i.e. $\mu_\varepsilon\in\tilde{\Sigma}_{n-1}^{\varepsilon}$ as specified in Definition~\ref{def:deg-spec}. 
We show that for a sufficiently small $\varepsilon$ and neighborhood $U_\varepsilon(\mu_\varepsilon)$ the number of zeros $\mu_\varepsilon$ of $\tilde \chi^\varepsilon _{n-1}$ counting multiplicity equals the number of eigenvalues $\lambda_\varepsilon\in \Sigma^\varepsilon \cap U_\varepsilon(\mu_\varepsilon)$. 
%We show that for a given $0<\delta<\delta_{n}$, there exits $\varepsilon_{n}>0$, such that for $0<\varepsilon<\varepsilon_{n}$ the number of eigenvalues in $\mathcal{B}_{\delta}(\lambda)$ counting multiplicities equal the number of eigenvalues $\mu\in\Sigma^{\varepsilon}\cap\mathcal{B}_{\delta}(\lambda)$ counting multiplicities. 
Again let the matrices $U_{n}=[U_{n,1},U_{n,2}]$ and $V_{n}=[V_{n,1},V_{n,2}]$ contain the left and right singular vectors corresponding to the cokernel ($U_{n,1}$ and $V_{n,1}$)
and image ($U_{n,2}$ and $V_{n,2}$) of $A_{n}$, see Def.~\ref{def:deg-spec} for details. Consider $z\in\mathbb{C}$, $|z|$ sufficiently small and define
\[
f_{\varepsilon}(z):=\chi^{\varepsilon}(z+\mu_\varepsilon)=\det\left(\begin{array}{cc}
C_{1}^{\varepsilon}(z) & C_{2}^{\varepsilon}(z)\\
C_{3}^{\varepsilon}(z) & C_{4}^{\varepsilon}(z)
\end{array}\right),
\]
 where
\begin{align*}
C_{1}^{\varepsilon}(z) & =-(z+\mu_\varepsilon)J_{1}^{(n)}+A_{0,1}^{(n)}+\sum\limits _{k=1}^{n-1}A_{k,1}^{(n)}e^{-\left(z+\mu_\varepsilon\right)\sigma_{k}\varepsilon^{-k}},\\
C_{2}^{\varepsilon}(z) & =-(z+\mu_\varepsilon)J_{2}^{(n)}+A_{0,2}^{(n)}+\sum\limits _{k=1}^{n-1}A_{k,2}^{(n)}e^{-\left(z+\mu_\varepsilon\right)\sigma_{k}\varepsilon^{-k}},\\
C_{3}^{\varepsilon}(z) & =-(z+\mu_\varepsilon)J_{3}^{(n)}+A_{0,3}^{(n)}+\sum\limits _{k=1}^{n-1}A_{k,3}^{(n)}e^{-\left(z+\mu_\varepsilon\right)\sigma_{k}\varepsilon^{-k}},\\
C_{4}^{\varepsilon}(z) & =-(z+\mu_\varepsilon)J_{4}^{(n)}+A_{0,4}^{(n)}+\sum\limits _{k=1}^{n-1}A_{k,4}^{(n)}e^{-\left(z+\mu_\varepsilon\right)\sigma_{k}\varepsilon^{-k}}+A_{n,4}^{(n)}e^{-\left(z+\mu_\varepsilon\right)\sigma_{n}\varepsilon^{-n}},
\end{align*}
is the block structure obtained from multiplying $\chi^{\varepsilon}(\lambda)$
by $\det U_{n}^{\ast}$ and $\det V_{n}$ from left and right with the corresponding
projected matrices
\begin{eqnarray*}
A_{k,1}^{(n)} & = & U_{n,1}^{\ast}A_{k}V_{n,1},\,J_{1}^{(n)}=U_{n,1}^{\ast}V_{n,1}\in\mathbb{C}^{(d-d_{n})\times(d-d_{n})},\\
A_{k,2}^{(n)} & = & U_{n,1}^{\ast}A_{k}V_{n,2},\,J_{2}^{(n)}=U_{n,1}^{\ast}V_{n,2}\in\mathbb{C}^{(d-d_{n})\times d_{n}},\\
A_{k,3}^{(n)} & = & U_{n,2}^{\ast}A_{k}V_{n,1},\,J_{3}^{(n)}=U_{n,2}^{\ast}V_{n,1}\in\mathbb{C}^{d_{n}\times(d-d_{n})},\\
A_{k,4}^{(n)} & = & U_{n,2}^{\ast}A_{k}V_{n,2},\,J_{4}^{(n)}=U_{n,2}^{\ast}V_{n,2}\in\mathbb{C}^{d_{n}\times d_{n}},
\end{eqnarray*}
for all $0\leq k\leq n-1.$ Using the Schur complement formula, we
obtain 
\[
f_{\varepsilon}(z)=\left(e^{-\left(z+\mu_\varepsilon\right)\sigma_{n}\varepsilon^{-n}}\right)^{d_{n}}\det\left(\tilde{C}_{1}^{\varepsilon}(z)\right)\det\left(\tilde{C}_{4}^{\varepsilon}(z)\right),
\]
where the matrices $\tilde{C}_{1}^{\varepsilon}(z)$ and ${C}_{4}^{\varepsilon}(z)$ are given by
\begin{align*}
\tilde{C}_{1}^{\varepsilon}(z)= & -(z+\mu_\varepsilon)J_{1}^{(n)}+A_{0,1}^{(n)}+\sum\limits _{k=1}^{n-1}A_{k,1}^{(n)}e^{-\left(z+\mu_\varepsilon\right)\sigma_{k}\varepsilon^{-k}}\\
 & -e^{\left(z+\mu_\varepsilon\right)\sigma_{n}\varepsilon^{-n}}C_{2}^{\varepsilon}(z)\left(\tilde{C}_{4}^{\varepsilon}(z)\right)^{-1}C_{3}^{\varepsilon}(z),\\
\tilde{C}_{4}^{\varepsilon}(z)= & A_{n,4}^{(n)}+e^{\left(z+\mu_\varepsilon\right)\sigma_{n}\varepsilon^{-n}}\left[-(z+\lambda)J_{4}^{(n)}+A_{0,4}^{(n)}+\sum\limits _{k=1}^{n-1}A_{k,4}^{(n)}e^{-\left(z+\mu_\varepsilon\right)\sigma_{k}\varepsilon^{-k}}\right].
\end{align*}
Note that
\[
\det\tilde{C}_{4}^{\varepsilon}(z)=\det A_{n,4}^{(n)}+\mathcal{O}\left(\left|e^{\left(z+\mu_\varepsilon\right)\varepsilon^{-1}}\right|\right),
\]
Choose $U^\varepsilon(\mu_\varepsilon)$ such that $\tilde \chi_{n-1}^{\varepsilon}(z+\mu_\varepsilon)\neq0$ and $\Re(z+\mu_\varepsilon)<0$ for all $z$ such that $z+\mu_\varepsilon \in U^\varepsilon(\mu_\varepsilon).$  
As a result, we have
 $$\left(\tilde{C}_{4}^{\varepsilon}(z)\right)^{-1}=\left(A_{n,4}^{(n)}\right)^{-1}+\mathcal{O}\left(\left|e^{\left(z+\mu_\varepsilon\right)\varepsilon^{-1}}\right|\right),$$
\[
\det\tilde{C}_{1}^{\varepsilon}(z)=\tilde \chi_{n-1}^{\varepsilon}(z+\mu_\varepsilon)+\mathcal{O}\left(\left|e^{\left(z+\mu_\varepsilon\right)\varepsilon^{-n}}\right|\right),
\]
and
\[
f_{\varepsilon}(z)\left(e^{\left(z+\mu_\varepsilon\right)\varepsilon^{-n}}\right)^{d_{n}}=\det(A_{k,4}^{(n)})\tilde{\chi}_{n-1}^{\varepsilon}(z+\mu_\varepsilon)+\mathcal{O}\left(\left|e^{\left(z+\mu_\varepsilon\right)\varepsilon^{-1}}\right|\right).
\]
where $\tilde \chi_{n-1}^{\varepsilon}(z+\mu_\varepsilon)$ is as in Definition~\ref{def:deg-spec}(i), and by assumption $\tilde \chi_{n-1}^{\varepsilon}(\mu_\varepsilon)=0$. 
The factor $(e^{\left(z+\mu_\varepsilon\right)\varepsilon^{-n}})^{d_{n}}$
remains bounded as $\varepsilon\to0$. Therefore, Rouch$\acute{\text{e}}$'s
Theorem implies that $f_{\varepsilon}$ has the same number of zeros as $\tilde{\chi}_{n-1}^{\varepsilon}(\cdot + \mu_\varepsilon)$ counting multiplicity. This proves the theorem for $k=n-1.$ If $\underline{k}<n-1$
this procedure has to be applied again to show that elements of $\tilde \Sigma_k^\varepsilon$ can be approximated by elements of  $\tilde \Sigma_{k-1}^\varepsilon$. The induction step $k\mapsto k-1$
obtaining $\tilde{\chi}_{k-1}^{\varepsilon}:\mathbb{C}\to\mathbb{R}$
and $\tilde{\Sigma}_{k-1}^{\varepsilon}$ is completely analogous for all $k\geq\underline{k}$. If $\underline{k}=1$,  we have to grantee that after the induction step $k=1\mapsto k=0$, the obtained truncated characteristic equation 
\begin{equation}\label{eq:chi-0-aux}
\tilde\chi_0(z)=\det\left(-zJ_{1}^{(1)}+A_{0,1}^{(1)}\right),
\end{equation}
is nontrivial, i.e. there exits $\mu$ such that $\tilde\chi_0(z)\neq0$. If $\det J_{1}^{(1)}\neq 0$, then $\tilde\chi_0(z)=0$ if and only if $z$ is an eigenvalue of the matrix $(J_{1}^{(1)})^{-1}A_{0,1}^{(1)}$, and hence $\tilde\chi_0(z)$ is nontrivial. Otherwise, let the matrices $\mathcal{U}=[\mathcal{U}_1,\mathcal{U}_2]$ and $\mathcal{V}=[\mathcal{V}_1,\mathcal{V}_2]$ contain the left and right singular vectors corresponding to the cokernel ($\mathcal{U}_1$ and $\mathcal{V}_1$) and image ($\mathcal{U}_2$ and $\mathcal{V}_2$) of $J_{1}^{(1)}$. Condition (ND) implies $\det(\mathcal{U}_1^\ast A_{0,1}^{(1)}\mathcal{V}_1)\neq0$. Thus, using the Schur complement formula, Eq.~(\ref{eq:chi-0-aux}) can be recast as
$$
\tilde\chi_0(z)=\det(\mathcal{U}_1^\ast A_{0,1}^{(1)}\mathcal{V}_1)g(z),$$
where
$$ g(z) = \det\left(-z\mathcal{U}_2^\ast J_{1}^{(1)} \mathcal{V}_2  + \mathcal{U}_2^\ast A_{0,1}^{(1)} \mathcal{V}_2
+ \mathcal{U}_1^\ast A_{0,1}^{(1)}\mathcal{V}_2 \left( \mathcal{U}_1^\ast A_{0,1}^{(1)}\mathcal{V}_1\right)^{-1} \mathcal{U}_2^\ast A_{0,1}^{(1)}\mathcal{V}_1\right).
$$
Thus, $\tilde\chi_0(z)=0$ if and only if $z$ is an eigenvalue of the matrix 
$$
\left(\mathcal{U}_2^\ast J_{1}^{(1)} \mathcal{V}_2\right)^{-1}  \left(\mathcal{U}_2^\ast A_{0,1}^{(1)} \mathcal{V}_2 + \mathcal{U}_1^\ast A_{0,1}^{(1)}\mathcal{V}_2 \left( \mathcal{U}_1^\ast A_{0,1}^{(1)}\mathcal{V}_1\right)^{-1} \mathcal{U}_2^\ast A_{0,1}^{(1)}\mathcal{V}_1\right)
$$
This proves the Theorem. 

\subsection*{Proof of Theorem \ref{thm:spec-approx}}

\subsubsection*{(i)}

Let $\lambda\in S_{0}^{+}$. For $z\in\mathcal{B}_{2\delta}(\lambda)$ the relation $\Re(z)>r>0$
holds. Hence, for $\varepsilon\to0$ and $z\in\mathcal{B}_{2\delta}(\lambda)$
the holomorphic function $\chi^{\varepsilon}(z)$ converges uniformly
to $\chi_{0}(z)$. The Hurwitz theorem implies that there exist $\varepsilon_{0}>0$
such that for $0<\varepsilon<\varepsilon_{0}$ the functions $\chi^{\varepsilon}(z)$
and $\chi_{0}(z)$ have the same number of zeros in $\mathcal{B}_{\delta}(\lambda)$.

\subsubsection*{(ii)}

This is an immediate consequence of Theorem~\ref{thm:degeneracy-spec}. The neighborhood $U^\varepsilon(\mu)$ can be chosen independent of $\varepsilon$. In particular, we can choose $U(\mu)=B_\delta(\mu)$ for all $\delta<r$, where $B_\delta(\mu)$ is the $\delta$-ball around $\mu$ and r is as defined in Def.~\ref{def:strong}.

\subsubsection*{(iii)}

At first, we introduce some necessary notation. For $\varepsilon>0$, we recursively define the integer valued functions $\Psi_{j}(\varepsilon),$
$j=1,\dots,n$ 
\begin{equation}
\Psi_{1}(\varepsilon):=\left[\frac{\omega_{0}}{2\pi\varepsilon}\right],\label{eq:psi1}
\end{equation}
\begin{equation}
\Psi_{j+1}:=\left[\frac{1}{\varepsilon}\left(\frac{\varphi_{j}}{2\pi}+\Psi_{j}\right)\right],\quad j=1,\dots,n-1,\label{eq:psidef}
\end{equation}
where $\left[\cdot\right]$ denotes the integer part. The following
lemma describes the properties of the functions $\Psi_{j}$, which
are necessary for our analysis.
\begin{lemma}
\label{lem:1}The following limits hold true
\begin{equation}
\lim_{\varepsilon\to0}\varepsilon^{k}\Psi_{k}(\varepsilon)=\frac{\omega_{0}}{2\pi}\label{eq:lim1}
\end{equation}
\begin{equation}
\lim_{\varepsilon\to0}\varepsilon^{j}\Psi_{k}(\varepsilon)=\frac{\varphi_{k-j}}{2\pi}\,\mathrm{mod}\,1,\quad\mbox{for}\,\,1\le j<k\le n.\label{eq:lim2}
\end{equation}
\end{lemma}
\begin{proof}
Firstly, the relation (\ref{eq:lim1}) follows from the following
\begin{equation}
\varepsilon\Psi_{1}(\varepsilon)=\varepsilon\left[\frac{\omega_{0}}{2\pi\varepsilon}\right]=\varepsilon\left(\frac{\omega_{0}}{2\pi\varepsilon}\right)+\mathcal{O}\left(\varepsilon\right)\to\frac{\omega_{0}}{2\pi}.\label{eq:11}
\end{equation}
Further for any $1\le j<k$, we have the following 
\[
\varepsilon^{j}\Psi_{k}(\varepsilon)=\varepsilon^{j}\left[\frac{1}{\varepsilon}\left(\frac{\varphi_{k-1}}{2\pi}+\Psi_{k-1}\right)\right]=\varepsilon^{j-1}\frac{\varphi_{k-1}}{2\pi}+\varepsilon^{j-1}\Psi_{k-1}+\mathcal{O}(\varepsilon^{j}).
\]
For brevity, we omit the arguments in $\Psi_{k-1}$ here and in the
following. For $j=1$, it follows that 
\[
\varepsilon\Psi_{k}(\varepsilon)\to\frac{\varphi_{k-1}}{2\pi}+\Psi_{k-1}=\frac{\varphi_{k-1}}{2\pi}\,\mbox{mod}\,1,
\]
which is a particular case of (\ref{eq:lim2}) for $j=1$. If $j>1$,
we have further 
\[
\varepsilon^{j}\Psi_{k}(\varepsilon)=\varepsilon^{j-1}\Psi_{k-1}+\mathcal{O}(\varepsilon^{j-1})=\varepsilon^{j-1}\left[\frac{1}{\varepsilon}\left(\frac{\varphi_{k-2}}{2\pi}+\Psi_{k-2}\right)\right]+\mathcal{O}(\varepsilon^{j-1})=
\]
\[
=\varepsilon^{j-2}\left(\frac{\varphi_{k-2}}{2\pi}+\Psi_{k-2}\right)+\mathcal{O}(\varepsilon^{j-1})=\cdots=\frac{\varphi_{k-j}}{2\pi}+\Psi_{k-j}+\mathcal{O}(\varepsilon)\to\frac{\varphi_{k-j}}{2\pi}\,\mbox{mod}\,1.
\]
This proves Eq. (\ref{eq:lim2}). Further, if $j=k$, we use (\ref{eq:lim2})
with $j=k-1$ and (\ref{eq:11}) to show 
\[
\varepsilon^{k}\Psi_{k}=\varepsilon\left(\frac{\varphi_{1}}{2\pi}+\Psi_{1}+\mathcal{O}(\varepsilon)\right)=\varepsilon\Psi_{1}+\mathcal{O}(\varepsilon)\to\frac{\omega_{0}}{2\pi}
\]
 which proves (\ref{eq:lim1}).
\end{proof}
We return to the proof of Theorem \ref{thm:spec-approx}\textit{(iii)}. Next, we show that for $1\leq k \leq n$, $\mu \in \mathcal{A}_k$ and $\delta>0$, there exits $\varepsilon_0>0$, such that for $0<\varepsilon<\varepsilon_0$,  there exists $\lambda\in\Sigma_{k,\delta}^{\varepsilon}\subset\Sigma^{\varepsilon}$
such that $|\Pi_{\varepsilon}^{(k)}(\lambda)-\mu|<\delta$. First, consider $k<n$ and $\mu=\gamma_{0}+i\omega_{0}\in S_{k}^{+}\subset\mathcal{A}_k$. For $\varepsilon>0$,
define $f_{\varepsilon}^{(k)}(z):=\chi^{\varepsilon}\left(\varepsilon^{k}(z+i2\pi\Psi_{k}(\varepsilon))\right)$,
i.e.
\begin{align}
f_{\varepsilon}^{(k)}(z)= & \det\left(-\varepsilon^{k}zI-i\varepsilon^{k}2\pi\Psi_{k}(\varepsilon)I+A_{0}+\sum_{j=1}^{k-1}A_{j}e^{-z\sigma_{j}\varepsilon^{k-j}-i2\pi\Psi_{k}(\varepsilon)\sigma_{j}\varepsilon^{k-j}}\right.\label{eq:fek}\\
 & \left.+A_{k}e^{-\sigma_{k}z}+\sum_{j=k+1}^{n}A_{j}e^{-z\sigma_{j}\varepsilon^{-j}-i2\pi\Psi_{k}(\varepsilon)\sigma_{j}\varepsilon^{-j}}\right)\nonumber 
\end{align}
For $\varepsilon\to0$, with the use of Lemma~\ref{lem:1}, we see
that 
\begin{align}
f_{\varepsilon}^{(k)}(z)\to&\det\left[-i\omega_{0}I+A_{0}+\sum_{j=1}^{k-1}A_{j}e^{-i\sigma_{j}\varphi_{j}}+A_{k}e^{-\sigma_{k}z}\right]\\
&=\chi_{k}\left(\omega_0,\varphi_{1},\dots,\varphi_{k-1};e^{-\sigma_{k}z}\right)\nonumber
\end{align}
locally uniformly for all $z$ with $\Re(z)>\delta$. Without loss
of generality, we assume $\delta<\gamma_{0}/2$. Let the polynomial $\chi_{k}$ be nontrivial. By assumption 
$\mu\in S_{k}^+$, such that there exists $\psi_{0}\in\mathbb{R}$ with
\[
\chi_{k}\left(\omega,\varphi_{1},\dots,\varphi_{k-1};e^{-\sigma_{k}(\gamma_{0}+i\psi_{0})}\right)=0.
\]
We can choose $\eta>0$, such that there exists $\varepsilon_{0}>0$ with the following property: For $0<\varepsilon<\varepsilon_{0}$,
$\chi_{k}$ and $f_{\varepsilon}^{(k)}$ have the same number of zeros
in the open $\eta$--disk of $\gamma_{0}+i\psi_{0}$. Here, $\eta>0$ is such that $\gamma_{0}+i\psi_{0}$ is the unique zero of $\chi_{k}$ in the closed disk around $\gamma_{0}+i\psi_{0}$ with radius $\eta$. If $z_{\varepsilon}$
is such a zero, then $\lambda_{\varepsilon}=\varepsilon^{k}z_{\varepsilon}+i\omega_0+\mathcal{O}(\varepsilon^{k+1})\in\Sigma^{\varepsilon}$.
Given $\delta>0$, we choose $\eta>0$ and $\varepsilon_{0}$ sufficiently
small such that $\mbox{dist}\left(\Pi_{\varepsilon}^{(k)}(\lambda_\varepsilon),S_{k}^+\right)<\delta$.

For the case $k=n$, we assume that $\mu\in S_{n}$. In this case,
define 
\begin{align*}
f_{\varepsilon}^{(n)}(z)= & \det\left( \vphantom{\sum_{j=1}^{k-1}A_{j,1}^{(k+1)}}
-\varepsilon^{n}zI-i\varepsilon^{n}2\pi\Psi_{n}(\varepsilon)I+A_{0}\right.\\
& \left.+\sum_{j=1}^{n-1}A_{j}e^{-z\sigma_{j}\varepsilon^{n-j}-i2\pi\Psi_{n}(\varepsilon)\sigma_{j}\varepsilon^{n-j}}+A_{n}e^{-\sigma_{k}z}\right)
\end{align*}
For $\varepsilon\to0$, $f_{\varepsilon}^{(n)}(z)\to\chi_{n}\left(\omega_0,\varphi_{1},\dots,\varphi_{n-1};e^{-\sigma_{n}z}\right)$
locally uniformly on $\mathbb{C}$. Then, similarly to the case with
$k<n$, there exists $\psi_{0}\in\mathbb{R}$ such that 
\[
\chi_{n}\left(\omega,\varphi_{1},\dots,\varphi_{n-1};e^{-\sigma_{k}(\gamma_{0}+i\psi_{0})}\right)=0,
\]
and again, if $\chi_n$ is nontrivial, we can choose $\eta>0$ such that $\chi_{n}$ has only $\gamma_{0}+i\psi_{0}$
as a zero on the closed disk around $\gamma_{0}+i\psi_{0}$. Then
there exists $\varepsilon_{0}>0$ such that for $0<\varepsilon<\varepsilon_{0}$,
$\chi_{n}$ and $f_{\varepsilon}^{(n)}$ have the same number of zeros
in the open disk of $\gamma_{0}+i\psi_{0}$. If $z_{\varepsilon}$
is such a zero, then $\lambda_{\varepsilon}=\varepsilon^{n}z_{\varepsilon}+i\omega_0+\mathcal{O}(\varepsilon^{n+1})\in\Sigma^{\varepsilon}$.
Given $\delta>0$, we choose $\eta>0$ and $\varepsilon_{0}$ sufficiently
small that $\mbox{dist}\left(\Pi_{\varepsilon}^{(n)}(\lambda),S_{n}\right)<\delta$. 
\\
The next case works analogous to the case $k=n$. We consider the case when $\det A_n =0$ and we have to consider non-generic spectrum up to some order $\underbar{k}<k<n$. Recall Def.~\ref{def:deg-spec}. We fix $k$  and consider  $\mu=\gamma_{0}+i\omega_{0}\in\tilde{S}_{k}^{-}$. Note $\Re(\mu)<0$ such that $|e^{-\sigma_{k}z}|$ is not a small perturbation. We define
 $\tilde{f}_{\varepsilon}^{(k)}(z):=\tilde{\chi}_{k}^{\varepsilon}\left(\varepsilon^{k}(z+i2\pi\Psi_{k}(\varepsilon))\right)$,
i.e.
\begin{align*}
\tilde{f}_{\varepsilon}^{(k)}(z) =& \det \left( \vphantom{\sum_{j=1}^{k-1}A_{j,1}^{(k+1)}} -\varepsilon^{k}zJ_{1}^{(k+1)}-i\varepsilon^{k}2\pi\Psi_{k}(\varepsilon)J_{1}^{(k+1)}+A_{0,1}^{(k+1)} \right.\\
 & \left.  +\sum_{j=1}^{k-1}A_{j,1}^{(k+1)}e^{-z\sigma_{j}\varepsilon^{k-j}-i2\pi\Psi_{k}(\varepsilon)\varepsilon^{k-j}}+A_{k,1}^{(k+1)}e^{-\sigma_{k}z}\right)
\end{align*}
For $\varepsilon\to0$, $\tilde{f}_{\varepsilon}^{(k)}(z)\to\tilde{\chi}_{k}\left(\omega_0,\varphi_{1},\dots,\varphi_{k-1};e^{-\sigma_{k}z}\right)$
locally uniformly on $\mathbb{C}$. Similarly, there exists $\psi_{0}\in\mathbb{R}$
such that 
\[
\tilde{\chi}_{k}\left(\omega,\varphi_{1},\dots,\varphi_{k-1};e^{-\sigma_{k}(\gamma_{0}+i\psi_{0})}\right)=0,
\]
and $\eta>0$ can be chosen such that $\tilde{\chi}_{k}$ has only
$\gamma_{0}+i\psi_{0}$ as a zero on the closed disk around $\gamma_{0}+i\psi_{0}$.
Then there exists $\varepsilon_{0}>0$ such that for $0<\varepsilon<\varepsilon_{0}$,
$\tilde{\chi}_{k}$ and $\tilde{f}_{\varepsilon}^{(k)}$ have the
same number of zeros in the $\eta$-open disk of $\gamma_{0}+i\psi_{0}$.
If $z_{\varepsilon}$ is such a zero, then $\lambda_{\varepsilon}=\varepsilon^{k}z_{\varepsilon}+i\omega_0+\mathcal{O}(\varepsilon^{k+1}) \in\tilde{\Sigma}_{k}^{\varepsilon}$.
Using Theorem~\ref{thm:degeneracy-spec} and given $\delta>0$, we
choose $\eta>0$ and $\varepsilon_{0}$ sufficiently small that $\mbox{dist}\left(\Pi_{\varepsilon}^{(k)}(\lambda_\varepsilon),\tilde{S}_{k}^{-}\right)<\delta$.

\subsubsection*{(iv) }

Assume \textit{(iv)} is false, then there exist $R_{0}>0$ and $\delta_{0}>0$ and $\lambda_{m}\in\Sigma_{c}^{\varepsilon_{m}}$, $m\in\mathbb{N}$, with
$\Im(\lambda_{m})\le R_{0}$, $\varepsilon_{m}>0$ and $\lim_{m\to\infty}\varepsilon_{m}=0$
such that
\begin{equation}
\left|\Re(\lambda_{m})\right|\ge\delta_{0}\,\,\mbox{for all}\,m\in\mathbb{N}\label{eq:contr1}
\end{equation}
or
\begin{equation}
\left|\Pi_{\varepsilon_{m}}^{(k)}(\lambda_{m})-\mu\right|\ge\delta_{0}\,\,\mbox{for all}\,\mu\in \mathcal{A}_{k},k\in\{1,\dots,n\}.\label{eq:contr2}
\end{equation}

Statement (\ref{eq:contr1}) is contradiction to the statement
of Lemma~\ref{lem:3} (below). We show that for any convergent subsequence $\left(\lambda_{m_{k}}\right)_{k\in\mathbb{N}}$
we have $$\Re\left(\lim_{k\to\infty}\lambda_{m_{k}}\right)=0.$$
Since by assumption, the imaginary parts are bounded, there
exists a subsequence converging to some $i\omega_{0}\in\mathbb{C}$.
Applying Lemma~\ref{lem:2} (below), there exist $1\leq k \leq n$ and $\mu\in\mathcal{A}_k$, such that $$\lim_{m\to\infty}\Pi_{\varepsilon_{j_{m}}}^{(k)}(\lambda_{j_{m}})\in\mathcal{A}_{k},$$
thereby contradicting (\ref{eq:contr2}). 

\begin{lemma}
\label{lem:2}Let $\left(\lambda_{j}\right)_{j\in\mathbb{N}}$ be
a sequence of complex numbers converging to $i\omega_{0}\in\mathbb{C}$,
where $\omega_{0}\in\mathbb{R}$, and let $\left(\varepsilon_{j}\right)_{j\in\mathbb{N}}$
be a sequence of positive numbers converging to zero such that $\chi^{\varepsilon_{j}}(\lambda_{j})=0$.
Then there exists a subsequence $\left(\lambda_{j_{m}}\right)_{j_{m}\in\mathbb{N}}$
such that one of the following holds:
\begin{itemize}
	\item[(a)] $\lim_{m\to\infty}\Pi_{\varepsilon_{j_{m}}}^{(k)}(\lambda_{j_{m}})\in\mathcal{A}_{k}$
	with some $1\le k\le n-1$. 
	\item[(b)] $\lim_{m\to\infty}\Pi_{\varepsilon_{j_{m}}}^{(n)}(\lambda_{j_{m}})\in \mathcal{A}_{n}$.
	\item[(c)] $\lim_{m\to\infty}\Pi_{\varepsilon_{j_{m}}}^{(k)}(\lambda_{j_{m}})=\infty$
	and $\lim_{m\to\infty}\Pi_{\varepsilon_{j_{m}}}^{(k-1)}(\lambda_{j_{m}})=i\omega_{0}$~where~$1\le k\le n$. 
	In this case, there exists a spectral manifold $\gamma_{l}^{(k)}$,
	$1\le l\le n$ such that $$\gamma_{l}^{(k)}(\omega_{0},\varphi_{1},\dots,\varphi_{k-1})=\infty$$
	for some $\varphi_{1},\dots,\varphi_{k-1}$. At the same time, $$\gamma_{l}^{(k-1)}(\omega_{0},\varphi_{1},\dots,\varphi_{k-2})=0$$
	in the case $k>1$. 
\end{itemize}

\end{lemma}
\begin{proof} Fix $1\leq k \leq n$ and write
\[
\lambda_{j}=\varepsilon_{j}^{k}\gamma_{j}^{(k)}+i\varepsilon_{j}^{k}\theta_{j}^{(k)}+i\varepsilon_{j}^{k}2\pi\Psi_{k}(\varepsilon),
\]
where $\theta_{j}^{(k)}\in[0,2\pi[$ and $\Psi_k$ as defined in Eqs.~(\ref{eq:psi1})--(\ref{eq:psidef}). Using Lemma~\ref{lem:1}, we have
$\lim_{j\to\infty}\varepsilon_{j}^{k}2\pi\Psi_{k}(\varepsilon)=\omega_{0}$. By assumption, it holds that 
$\lim_{j\to\infty}\varepsilon_{j}^{k}\gamma_{j}^{(k)}=0$. Passing to the subsequence, we can assume that $\lim_{j\to\infty}\theta_{j}^{(k)}=\theta_{0}^{(k)}\in[0,2\pi[$.
We define
\begin{align*}
\rho_{j}^{(k)}(y):= & \det\left(-\varepsilon_{j}^{k}yI-i\varepsilon_{j}^{k}2\pi\Psi_{k}(\varepsilon_{j})I+A_{0}+\sum_{l=1}^{k-1}A_{l}e^{-y\sigma_{l}\varepsilon_{j}^{k-l}-i2\pi\Psi_{k}(\varepsilon)\sigma_{l}\varepsilon_{j}^{k-l}}\right.\\
 & \qquad\quad \left.+A_{k}e^{-\sigma_{k}y}+\sum_{l=k+1}^{n}A_{l}e^{-y\sigma_{l}\varepsilon_{j}^{k-l}-i2\pi\Psi_{k}(\varepsilon_{j})\sigma_{l}\varepsilon_{j}^{k-l}}\right)
 =f_{\varepsilon_{j}}^{(k)}(y),
\end{align*}
where $f_{\varepsilon}^{(k)}(y)$ is as in Eq.~(\ref{eq:fek}).
Note that
\begin{equation}
\rho_{j}^{(k)}(\gamma_{j}^{(k)}+i\theta_{j}^{(k)})=\chi^{\varepsilon_{j}}(\lambda_{j})=0.\label{eq:rho_chi}
\end{equation}
Similarly, define
\begin{align*}
\tilde{\rho}_{j}^{(k)}(y):=\tilde{f}_{\varepsilon_{j}}^{(k)}(y)= & \det\left(-\varepsilon_{j}^{k}yJ_{1}^{(k+1)}-i\varepsilon_{j}^{k}2\pi\Psi_{k}(\varepsilon_{j})J_{1}^{(k+1)}+A_{0,1}^{(k+1)}\right.\\
 & \left.+\sum_{l=1}^{k-1}A_{l,1}^{(k+1)}e^{-y\sigma_{l}\varepsilon_{j}^{k-l}-i2\pi\Psi_{k}(\varepsilon)\sigma_{l}\varepsilon_{j}^{k-l}}+A_{k,1}^{(k+1)}e^{-\sigma_{k}y}\right),
\end{align*}
such that
\begin{equation}
\tilde{\rho}_{j}^{(k)}(\gamma_{j}^{(k)}+i\theta_{j}^{(k)})=\tilde{\chi}_{k}^{\varepsilon_{j}}(\lambda_{j})=0.\label{eq:rho_chi-1}
\end{equation}

For $k=n$ the sequence of holomorphic functions $\rho_{j}^{(n)}(y)$
converges uniformly on bounded sets of $\mathbb{C}$ to $\chi_{n}\left(\omega,\varphi_{1},\dots,\varphi_{n-1};e^{-\sigma_{n}y}\right)$,
and for $k<n$ on bounded sets of $\left\{ \lambda\in\mathbb{C}\,|\,\Re(\lambda)>\delta\right\} $
to $\chi_{k}\left(\omega,\varphi_{1},\dots,\varphi_{k-1};e^{-\sigma_{k}y}\right)$,
with any $\delta>0$. Similarly, the sequence of holomorphic functions
$\tilde{\rho}_{j}^{(n)}(y)$ converges uniformly on bounded sets of
$\left\{ \lambda\in\mathbb{C}\,|\,\Re(\lambda)<\delta\right\} $ to $\tilde{\chi}_{k}\left(\omega,\varphi_{1},\dots,\varphi_{n-1};e^{-\sigma_{k}y}\right)$.

Let $\tilde{k}$ be the largest number between $1$ and $n$ such
that the sequence $\gamma_{j}^{(\tilde{k})}$ is bounded. If no such
$\tilde{k}$ exists, the sequence is $\gamma_{j}^{(1)}$ is unbounded
and this case will be considered later.

\textit{Case (b): $\tilde{k}=n$.} There exists a subsequence $(\gamma_{j_{m}}^{(\tilde{k})})_{m\in\mathbb{N}}$
converging to $\gamma_{0}^{(\tilde{k})}\in\mathbb{R}$. Letting $m\to\infty$,
we have $$\chi_{n}\left(\omega_{0},\varphi_{1},\dots,\varphi_{n-1};e^{-\sigma_{n}(\gamma_{0}^{(\tilde{k})}+i\theta_{0}^{(\tilde{k})})}\right)=0.$$
and therefore $$\gamma_{0}^{(\tilde{k})}+i\theta_{0}^{(\tilde{k})}==\lim_{m\to\infty}\Pi_{\varepsilon_{j_{m}}}^{(n)}(\lambda_{j_{m}})\in S_{n}.$$ This implies \textit{(b)}. 

\textit{Case (a): $\tilde{k}<n$.} Then there exists a subsequence $(\gamma_{j_{m}}^{(\tilde{k})})_{m\in\mathbb{N}}$
converging to $\gamma_{0}^{(\tilde{k})}\in\mathbb{R}$. For $\gamma_{0}^{(\tilde{k})}>0$, 
we choose $\delta=\gamma_{0}/2$ and letting $m\to\infty$, we have either $$\chi_{k}\left(\omega_{0},\varphi_{1},\dots,\varphi_{k-1};e^{-\sigma_{k}(\gamma_{0}^{(\tilde{k})}=i\theta_{0}^{(\tilde{k})})}\right)=0$$
or  $$\tilde{\chi}_{k}\left(\omega_{0},\varphi_{1},\dots,\varphi_{k-1};e^{-\sigma_{k}(\gamma_{0}^{(\tilde{k})}+i\theta_{0}^{(\tilde{k})})}\right)=0.$$
Hence, in this case we have $$\gamma_{0}^{(\tilde{k})}+i\theta_{0}^{(\tilde{k})}=\lim_{m\to\infty}\Pi_{\varepsilon_{j_{m}}}^{(k)}(\lambda_{j_{m}})\in\mathcal{A}_{\tilde{k}}.$$
This implies \textit{(a)}. If $\gamma_{0}^{(\tilde{k})}=0$, then Theorem \ref{thm:spec-mnf}(iii) implies that
$\gamma_{j}^{(\tilde{k}+1)}$ is unbounded. This case will be considered later.\\
\textit{Case (c):} We study the case when the sequence $\gamma_{j}^{(1)}$ is unbounded
and $|\gamma_{j_{m}}^{(1)}|\to\infty$. Using (\ref{eq:rho_chi})--(\ref{eq:rho_chi-1}),
it follows that $\chi_{1}(\omega;0)=0$, or $\tilde{\chi}_{1}(\omega;0)=0$ as $m\to\infty$. We can assume that $\chi_{1}(\omega;0)=0$ (or respectively $\tilde{\chi}_{1}(\omega;0)=0$) is nontrivial. Therefore, there exists a spectral curve $\gamma_{l}^{(1)}(\omega)$
 with some $l\in\{1,\dots,n\}$ such that $|\gamma_{l}^{(1)}(\omega_{0})|=\infty $.
Now consider the case when $\gamma_{0}^{(\tilde{k})}=0$ and the sequence $\gamma_{j}^{(\tilde{k}+1)}$
is unbounded. Let $\tilde{k}<n$. If $$|\gamma_{j_{m}}^{(\tilde{k}+1)}|\to\infty$$
it follows that $$\chi_{\tilde{k}+1}\left(\omega,\varphi_{1},\dots,\varphi_{\tilde{k}};0\right)=0
\quad\mbox{or}\quad 
\tilde{\chi}_{\tilde{k}+1}\left(\omega,\varphi_{1},\dots,\varphi_{\tilde{k}};0\right)=0$$
with $\varphi_{\tilde{k}}=\theta_{0}^{(\tilde{k})}$. Then there exists
a spectral manifold $\gamma_{l}^{(\tilde{k}+1)}$, $1\le l\le n$
such that $$|\gamma_{l}^{(\tilde{k}+1)}(\omega_{0},\varphi_{1},\dots,\varphi_{\tilde{k}})|=\infty $$
for some $\varphi_{1},\dots,\varphi_{\tilde{k}}$. From Lemma~(\ref{lem:1})
it follows that $\gamma_{l}^{(\tilde{k})}(\omega_{0},\varphi_{1},\dots,\varphi_{\tilde{k}-1})=0$.
Moreover, $$\lim_{m\to\infty}|\Pi_{\varepsilon_{j_{m}}}^{(\tilde{k}+1)}(\lambda_{j_{m}})|=\infty $$
and $\lim_{m\to\infty}\Pi_{\varepsilon_{j_{m}}}^{(\tilde{k})}(\lambda_{j_{m}})=i\omega_{0}$.
This implies \textit{(c)}. 
\end{proof}
\begin{lemma}
\label{lem:3}Let $R>0$. For $n\in\mathbb{N}$, let $\varepsilon_{n}>0$
be such that $\lim_{m\to\infty}\varepsilon_{n}=0$. Consider $\lambda_{n}\in\Sigma_{c}^{\varepsilon_{n}}$
with $\left|\Im(\lambda_{n})\right|\le R$. Then $\lambda_{n}$ is
bounded, and for any convergent subsequence $\left(\lambda_{n_{k}}\right)_{k\in\mathbb{N}}$
we have $\Re\left(\lim_{k\to\infty}\lambda_{n_{k}}\right)=0$.
\end{lemma}
\begin{proof}
Let us show that $\lambda_{n}$ is bounded. For this assume the opposite,
i.e. there exists a subsequence $\left(\lambda_{n_{k}}\right)_{k\in\mathbb{N}}$
such that either 
\begin{equation}
\lim_{k\to\infty}\Re\left(\lambda_{n_{k}}\right)=\infty\label{eq:infty}
\end{equation}
or 
\begin{equation}
\lim_{k\to\infty}\Re\left(\lambda_{n_{k}}\right)=-\infty.\label{eq:_infty}
\end{equation}
In the case (\ref{eq:infty}) the characteristic equation (\ref{eq:cheq})
has the following asymptotics \linebreak $\chi^{\varepsilon_{n_{k}}}(\lambda_{n_{k}})=(-1)^{n}\lambda_{n_{k}}^{n}+\mathcal{O}(1),$
which is clearly nonzero for all large enough $k$. Hence (\ref{eq:infty})
is not possible. In case (\ref{eq:_infty}), the leading term of the
characteristic equation $e^{-n\sigma_{n}\lambda_{n_{k}}\varepsilon^{-n}}\det\left(A_{n}\right)$
is not zero for large enough $k>k_{0}$. Thus, we arrive at the contradiction
to (\ref{eq:_infty}). Hence, $\lambda_{n}$ is bounded. \\
Let $\left(\lambda_{n_{k}}\right)_{k\in\mathbb{N}}$ be any subsequence
converging to $\lambda_{0}$. Suppose $\Re(\lambda_{0})>0$. Then,
passing to the limit in (\ref{eq:cheq}), we obtain $\det\left[-\lambda_{0}I+A_{0}\right]=0$,
which contradicts to the assumption $\lambda_{n_{k}}\not\in\Sigma_{0}^{\varepsilon_{n_{k}}}$.
Suppose $\Re \left( \lambda_0 \right)<0$. Then, Theorem \ref{thm:degeneracy-spec} implies that there is $\mu \in S_0^-$, such that $\lambda_{n_{k}}\to\mu$ and we again
arrive at the contradiction to $\lambda_{n_{k}}\not\in\Sigma_{0}^{\varepsilon_{n_{k}}}$. Hence, $\Re\left(\lim_{k\to\infty}\lambda_{n_{k}}\right)=0$. 
\end{proof}

\subsection*{Proof of Theorem \ref{thm:spec-mnf} }

Let $1\leq k\leq n$ and $(\omega,\varphi_1,\ldots,\varphi_{k-1})\in\mathbb{R}\times\mathbb{S}^{k-1}$ be fixed. The truncated characteristic equation $\chi_k(\omega,\varphi_1,\ldots,\varphi_{k-1}; Y)=0$ is a complex polynomial in $Y$ of degree $d_{k}=\text{rank}A_{k}$ with roots $Y_{l}^{(k)}$, $l=1,\ldots,d_{k}$. 

These roots depend continuously on $(\omega,\varphi_1,\ldots,\varphi_{k-1})$. Hence, there are $d_k$ continuous functions $Y_{l}^{(k)}:(\omega,\varphi_{1},\dots,\varphi_{k-1})\mapsto Y_{l}^{(k)}(\omega,\varphi_{1},\dots,\varphi_{k-1}),\mathbb{R}^{k}\to\mathbb{C},$ and $$\gamma_{l}^{(k)}(\omega,\varphi_{1},\dots,\varphi_{k-1}):=-\frac{1}{\sigma_k}\ln\left|Y_{l}^{(k)}(\omega,\varphi_{1},\dots,\varphi_{k-1})\right|,$$
such that 
$$ \chi_k(\omega,\varphi_1,\ldots,\varphi_{k-1}; Y_{l}^{(k)}(\omega,\varphi_{1},\dots,\varphi_{k-1}))=0.$$
Let $S_k$ be defined as in Def.~\ref{def:trunc-spec}. This proves \textit{(i)}. 

Statements (ii) and (iii) characterize $(\omega,\varphi_1,\ldots,\varphi_{k-1})\in\mathbb{R}\times\mathbb{S}^{k-1}$ such that $$\left|Y_{l}^{(k)}(\omega,\varphi_{1},\dots,\varphi_{k-1})\right|\in\{0,\infty\}$$ corresponding to the situation when $\gamma_{l}^{(k)}(\omega,\varphi_{1},\dots,\varphi_{k-1})$ is singular. 

\subsubsection*{(ii) }

Let $U_{k,1}$ and $V_{k,1}$ ($U_{k,2}$ and $V_{k,2}$) be the matrices
containing the left and right singular vectors of $A_{k}$ corresponding
to the singular value zero (to the nonzero singular values). Let $\tilde{A}_{k}=U_{k,2}^{\ast}A_{k}V_{k,2}$.
Then, it holds that
\[
\chi_{k}\left(\omega,\varphi_{1},\dots,\varphi_{k-1};Y\right)=\det\left(\begin{array}{cc}
U_{k,1}^{\ast}B_{k}V_{k,1} & U_{k,1}^{\ast}B_{k}V_{k,2}\\
U_{k,2}^{\ast}B_{k}V_{k,1} & U_{k,2}^{\ast}B_{k}V_{k,2}+\tilde{A}_{k}Y
\end{array}\right)
\]
where we omitted the arguments of
\begin{equation}\label{eq:matrix-Bk}
B_{k}=B_{k}(\omega,\varphi_{1},\dots,\varphi_{k-1})=-i\omega I+A{}_{0}+\sum\limits _{j=1}^{k-1}A_{j}e^{-i\sigma_j\varphi_{j}}.
\end{equation}
We apply the Schur decomposition formula and develop the determinant with respect to the columns of $\tilde{A}_{k}$ to see that the leading order monomial of the polynomial  $\chi_{k}\left(\omega,\varphi_{1},\dots,\varphi_{k-1};Y\right)$
is $\det(U_{k,1}^{\ast}B_{k}V_{k,1})\det\tilde{A}_{k}Y^{d_{k}}$, i.e.
\[
\det\left(U_{k,1}^{\ast}\left(-i\omega I+A{}_{0}+\sum\limits _{j=1}^{k-1}A_{j}e^{-i\sigma_j\varphi_{j}}\right)V_{k,1}\right)\det\tilde{A}_{k}Y^{d_{k}},
\]
the coefficient of which is non-zero by assumption. As a result, $$\gamma_{l}^{(k)}(\omega,\varphi_{1},\dots,\varphi_{k-1})=\infty$$
for some $l$ if and only if $\chi_{k}\left(\omega,\varphi_{1},\dots,\varphi_{k-1};0\right)=0,$ i.e.
$$
\det\left(-i\omega I+A_{0}+\sum_{j=1}^{k-1}A_{j}e^{-i\sigma_j\varphi_{j}}\right)=0.
$$
The last assertion of \textit{(ii)} follows from the fact that for $1\leq k < n$
\[
\chi_{k}\left(\omega,\varphi_{1},\dots,\varphi_{k-1};e^{i\varphi_{k}}\right)=\chi_{k+1}\left(\omega,\varphi_{1},\dots,\varphi_{k-1},\varphi_{k};0\right).
\]

\subsubsection*{(iii)}
%Assume that $i\omega\notin\sigma(A_{0}+\sum\limits _{j=1}^{k-1}A_{j}e^{-i\sigma_j\varphi_{j}})$.
In order to study the case when $|Y_{l}^{(k)}(\omega,\varphi_{1},\dots,\varphi_{k-1})|$ is unbounded, denote 
\[
q_{k}\left(\omega,\varphi_{1},\dots,\varphi_{k-1};Z\right):=\det\left[Z\left(\begin{array}{cc}
U_{k,1}^{\ast}B_{k}V_{k,1} & U_{k,1}^{\ast}B_{k}V_{k,2}\\
U_{k,2}^{\ast}B_{k}V_{k,1} & U_{k,2}^{\ast}B_{k}V_{k,2}
\end{array}\right)+\left(\begin{array}{cc}
0 & 0\\
0 & \tilde{A}_{k}
\end{array}\right)\right]
\]
where $B_k$ is as in Eq.~(\ref{eq:matrix-Bk}), and $B_k$ is invertible by assumption
\begin{equation}\label{eq:q-blockstructure}
\det\left(-i\omega I+A_{0}+\sum_{j=1}^{k-1}A_{j}e^{-i\sigma_j\varphi_{j}}\right)\ne0.
\end{equation}
 For $Z\ne0$, we have $q_{k}(\omega,\varphi_{1},\dots,\varphi_{k-1};Z)=0$
if and only if $$\chi_{k}(\omega,\varphi_{1},\dots,\varphi_{k-1};1/Z)=0$$ and hence,
we study roots of 
\begin{equation} \label{eq:q_k-aux}
q_{k}(\omega,\varphi_{1},\dots,\varphi_{k-1};Z)=0
\end{equation}
that tend to zero. Since $\mbox{rank}A_{k}=d_{k}<d$, $Z=0$ is a root of Eq.~(\ref{eq:q_k-aux}) with
multiplicity $d-d_{k}$. Denote $Q_k(Z)=\det\tilde A_k+U_{k,2}^{\ast}B_{k}V_{k,2}$. For $|Z|$ sufficiently small, we have $\det Q_{k}(Z)=\det\tilde A_k+\mathcal{O}(Z),$ and therefore $Q_{k}^{-1}(Z)$ is invertible with $Q_{k}^{-1}(Z)=\tilde A_k^{-1}+\mathcal{O}(Z).$ The following Lemma seperates the nontrivial component $\tilde q_k$ from $q_k$. The Lemma is a direct consequence of the Schur complement formula applied to Eq.~(\ref{eq:q-blockstructure}). We omit the details here. 
\begin{lemma}[separating nontrivial component of $q$]\label{lemma:nontrivial}
For all $\left(\omega,\varphi_{1},\dots,\varphi_{k-1}\right)\in\mathbb{R}^{k}$
and $|Z|$ sufficiently small, we have
\[
q_{k}(Z)=Z^{d-d_k}\det(Q_{k}(Z))\tilde{q}_{k}(Z),
\]
where we omit the dependency on $\left(\omega,\varphi_{1},\dots,\varphi_{k-1}\right)$
for brevity. Moreover, 
\[
\tilde{q}_{k}(Z)=\det\left(U_{k,1}^{\ast}B_{k}V_{k,1}-Z\left(U_{k,1}^{\ast}B_{k}V_{k,2}\right)Q_{k}(Z)^{-1}\left(U_{k,2}^{\ast}B_{k}V_{k,1}\right)\right).
\] 

\end{lemma}
Now as an immediate consequence of Lemma \ref{lemma:nontrivial},
$\gamma_{l}^{(k)}(\omega,\varphi_{1},\dots,\varphi_{k-1})=-\infty$ for some $l$, if and only if $\tilde{q}_{k}\left(\omega,\varphi_{1},\dots,\varphi_{k-1};0\right)=0$
implying that $\det\left(U_{k,1}^{\ast}B_{k}V_{k,1}\right)=0$, and therefore
\[
\det\left(U_{k,1}^{\ast}\left(-i\omega I+A{}_{0}+\sum\limits _{j=1}^{k-1}A_{j}e^{-i\varphi_{j}}\right)V_{k,1}\right)=0.
\]
This proves assertion \textit{(iii)}.

\section{Conclusions}
This article provides a rigorous description of the spectrum of DDEs with constant delays acting on different time scales, i.e. constant hierarchical delays of the form $\tau_k \sim \varepsilon^k$, $k=1,\ldots,n$ and $\varepsilon>0$ small. Such a scale separation between the delays allows for an explicit decomposition of the spectrum and the decomposition reflects the hierarchical structure of the delays. Each component of the decomposition can be approximated by relatively simple sets that can be computed explicitly in many cases. The particular scaling of delays may appear very special in terms of applications, yet the coefficients $\sigma_k>0$ in our ansatz  $\tau_k = \sigma_k\varepsilon^{-k}$ grant a certain freedom in the scaling assumption. Delays with such scaling properties have been previously implemented experimentally in systems of coupled semiconductor lasers \cite{Marconi2015, Giacomelli2019}. Previous studies had already shown that \cite{Yanchuk2014a,Yanchuk2015} that DDEs with hierarchical delays possess interesting dynamical properties, which resemble those of Partial Differential Equations in several spatial dimensions. Therefore, rigorous results concerning their spectrum play important role in the study of such systems. 

We have extended the results obtained in \cite{Lichtner2011} for a single large delay to multiple large hierarchical delays, and under more general non-genericity conditions. The non-genericity condition on $A_{1}$ in \cite{Lichtner2011}, i.e. $\ker A_{1}=\ker A_{1}^{2}$, as the matrix of highest order with respect to $1/\varepsilon$ has been replaced by the abstract rank condition (ND). In particular, our results hold true when $n=1$ and $\dim\ker A_{1}< \dim\ker A_{1}^{2}$. 

Many open questions remain: 
Condition (ND) is not yet well understood and so are the algebraic
properties of the spectral manifolds. Even the case $d=2,\,n=1$ needs
to be studied in more detail; here, two spectral manifolds can 'merge'
as the solutions to Eq.~(\ref{eq:cheq-k}) undergo a complex fold bifurcation.
On another note, the explicit algorithm to compute
the degeneracy spectrum is still missing. The correspoding iteration scheme can be
derived from the proof of Theorem~\ref{thm:degeneracy-spec}. At the same time, 
in order to estimate the minimum required 'largeness' of the
delays and their minimum scale separation $\sigma_{k},1\leq k\leq n$
for Theorems~\ref{thm:degeneracy-spec} and \ref{thm:spec-approx},
one should provide necessary (if not sufficient) conditions
for the convergence of the numerical methods used in Figs.~\ref{fig:spec-1} and \ref{fig:spec-2}. 

Ultimately, several works point towards the fact, that our results can be generalized to non-autonomous DDEs. Spectral splitting of non-autonomous DDEs of has been observed for the Lyapunov spectrum of DDEs with a single large delay \cite{Sieber2013a,Heiligenthal2011} and two large delays \cite{DHuys2013}, as well as for the  Floquet spectrum of periodic orbits in DDEs with a single large delay \cite{Sieber2013a,Yanchuk2019}. All of the points listed above present interesting problems to be addressed in future research. 

%\bibliographystyle{AIMS}
%\bibliography{Mult_Hier_Delay}

\end{document}